\documentclass[journal]{IEEEtran}
\usepackage{tikz}
\usepackage{amsfonts}
\usepackage{scrextend}
\usepackage{amsfonts}
\usepackage{amsmath}
\usepackage{amssymb}
\usepackage{enumitem}
\usepackage{mathtools}
\usepackage{amsthm}
\usepackage{breqn}
\usepackage{dsfont}
\allowdisplaybreaks

   \newtheorem{theorem}{\bf{Theorem}}

       \newtheorem{lemma}{\bf{Lemma}}
       \newtheorem{proposition}{\bf{Proposition}}
       \newtheorem{definition}{\bf{Definition}}
  \newtheorem{remark}{Remark}
  \newtheorem{assumption}{Assumption}

\newcommand\scalemath[2]{\scalebox{#1}{\mbox{\ensuremath{\displaystyle #2}}}}
\begin{document}
\title{Optimal Solutions to Infinite-Player Stochastic Teams and Mean-Field Teams}

\author{Sina~Sanjari,~\IEEEmembership{Student Member,~IEEE,}
        Serdar Y\"uksel,~\IEEEmembership{Member,~IEEE,}
\thanks{This research was supported by the Natural Sciences and Engineering Research Council (NSERC) of Canada. A summary of some of the results here, without technical details, was presented in \cite{sinacdc2018optimal} at the 57th IEEE Conference on Decision and Control (CDC), Miami Beach, FL, 2018. The authors are with the Department of Mathematics and Statistics,
     Queen's University, Kingston, ON, Canada,
     Email: \{16sss3,yuksel@queensu.ca\}}
     }

\maketitle

\begin{abstract}
We study stochastic static teams with countably infinite number of decision makers, with the goal of obtaining (globally) optimal policies under a decentralized information structure. We present sufficient conditions to connect the concepts of team optimality and person by person optimality for static teams with countably infinite number of decision makers. We show that under uniform integrability and uniform convergence conditions, an optimal policy for static teams with countably infinite number of decision makers can be established as the limit of sequences of optimal policies for static teams with $N$ decision makers as $N \to \infty$. Under the presence of a symmetry condition, we relax the conditions and this leads to optimality results for a large class of mean-field optimal team problems where the existing results have been limited to person-by-person-optimality and not global optimality (under strict decentralization). In particular, we establish the optimality of symmetric (i.e., identical) policies for such problems. As a further condition, this optimality result leads to an existence result for mean-field teams. We consider a number of illustrative examples where the theory is applied to setups with either infinitely many decision makers or an infinite-horizon stochastic control problem reduced to a static team. 
\end{abstract}


% \fi
\IEEEpeerreviewmaketitle

\section{Introduction}
\label{sec:intro}

A decentralized control system, or a team, consists of a collection of decision makers/agents acting together to optimize a common cost function, but not necessarily sharing all the available information. Teams whose initial states, observations, cost function, or the evolution dynamics are random or are disturbed by some external noise processes are called \textit{stochastic teams}. At each time stage, each agent only has access to some parts of the global information. If each agent's information depends only on primitive random variables, the team is \textit{static}. If at least one agent's information is affected by an action of another agent, the team is said to be \textit{dynamic}. 

%A team in which there is a pre-specified order of actions is said to be a \textit{sequential} team. {A sequential dynamic team in which the information available to the decision makers in forward time is {\it expanding} is called {\it classical}, i.e., in each time stage the information of decision makers contains that available information to the previous decision maker}. If whenever an agent's, say $\bf{A}$k's, information is affected by the action of some other agent, say $\bf{A}$j, $\bf{A}$k has access to $\bf{A}$j's information, the sequential team is \textit{quasi-classical} (or \textit{partially nested}) \cite{HoChu}. If an information structure is not quasi-classical, we call it \textit{non-classical} \cite{YukselBasarBook}, see also \cite{WitsenStandard}. %The main contributions of this paper are for static teams; however, the result is also useful for those static reduction of dynamics teams \cite{WitsenStandard, HoChu}.\par

On teams with finitely many decision makers, Marschak \cite{mar55} studied optimal static teams and Radner \cite{Radner} developed foundational results on optimality and established connections between person-by-person optimality, stationarity, and team-optimality. Radner's results were generalized in \cite{KraMar82} by relaxing optimality conditions. A summary of these results is that in the context of static team problems,  convexity of the cost function, subject to minor regularity conditions, may suffice for the global optimality of person-by-person-optimal solutions. In the particular case for LQG (Linear Quadratic Gaussian) static teams, this result leads to the optimality of linear policies \cite{Radner}, which also applies for dynamic LQG problems under specific information structures (to be discussed further below) \cite{HoChu}. These results are applicable for static teams with finite number of decision makers. In our paper, the focus is on teams with infinitely many decision makers.

{\bf Connections with the literature on mean-field games/teams.} On the case with infinitely many decision makers, a related set of results involves mean-field games: mean-field games (see e.g., \cite{CainesMeanField2,CainesMeanField3,LyonsMeanField}) can be viewed as limit models of symmetric non-zero-sum non-cooperative $N$-player games with a mean-field interaction as $N\to \infty$. The uniqueness and non-uniqueness results have been established for mean-field games in both the PDE and probabilistic setting \cite{LyonsMeanField, bardi2017non, carmona2016mean}. In \cite{bardi2017non}, examples have been provided to show the existence of multiple solutions to the mean-field games when uniquness conditions in \cite{LyonsMeanField, carmona2016mean} are violated.  The mean-field approach designs policies for both cases of games with infinitely many players, as well as games with very large number of players where the equilibruim policies for the former are shown to be $\epsilon$-equilibria for the latter \cite{CainesMeanField3, saldi2019approximate, cecchin2018probabilistic}. These results, while very useful for establishing  equilibria or in the context of team problems, person-by-person-optimal policies, does not guarantee the $\epsilon$-global optimality among all policies. That is, $\epsilon$-person-by-person-optimality is not sufficient for $\epsilon$-global optimality since in the limit one typically only finds equilibrium policies without establishing their uniqueness (which would imply global optimality for team problems) \cite{mas1984theorem, schmeidler1973equilibrium, jovanovic1988anonymous}. Related to such problems, in the economic theory literature, \cite{schmeidler1973equilibrium,mas1984theorem}, have considered {\it Cournot-Nash} equilibria. 
%, where in a non-cooperative game each player responds to infinitely many players through their (expected empirical) distribution and since players have identical utilities and an individual player is macroscopically negligible (hence the name {\it anonymous games} is also used for such setups), equilibrium solutions are established under possibly randomized strategies. 
This Cournot-Nash equilibrium concept corresponds to a mean-field equilibrium for a static problem. However, such an equilibrium does not necessarily imply global optimality in the context of team problems, as discussed above.

Recently, mean-field team problems have also been studied: Social optima for mean-field LQG control problems under both centralized and a specific decentralized information structure have been considered in \cite{huang2012social, wang2017social}. In \cite{arabneydi2015team}, a setup is considered where decision makers share some information on the mean-field in the system, and through showing that the performance of a corresponding centralized system can be realized under a decentralized information structure, global optimality is established. In our paper, we follow an approach where optimality for every $N$ is established and also optimality holds as $N\to \infty$ for the limit policy. The papers \cite{huang2016linear,huang2011stochastic} have studied a continuous-time setup where a major agent is present; by considering the social impact for each individual player, they showed person-by-person optimal policies asymptotically minimize the social cost \cite{huang2012social}. By approximating the mean-field term, the authors bound the induced approximation error of order $O(N^{\frac{-1}{2}}+\epsilon_{N})$ where $\epsilon_{N}$ goes to zero as the number of players $N \to \infty$ \cite{huang2012social}. In \cite{buckdahn2014nonlinear}, mean-field team problems with mixed players have been considered where minor agents act together to minimize a common cost against a major player. Also, for the LQ setup, under the assumption that DMs apply identical policy in addition to some technical assumptions on the cost function and transition probabilities of Markov chains, \cite{arabneydi2017certainty} showed that the expected cost achieved by a sub-optimal fully decentralized strategy is on $\epsilon(n)$ neighborhood of the optimal cost achieved when mean-field (empirical distribution of states) has been shared, where $n$ is the number of players. Such results on mean-field teams either show global optimality through equivalence to the performance of a centralized setup (considering specific sharing patterns on the mean-field model) or typically only assume person-by-person-optimality. In our paper, we will establish global optimality under a completely decentralized information structure; however, certain technical conditions will be imposed.

{\bf Connections with the literature on limits of finite player games/teams.} There exist contributions where games with finitely many players are studied, their equilibrium solutions are obtained and the limit is taken. Along this direction, the connection between Nash equilibrium of symmetric $N$-player games and an optimal solution of mean-field games has been addressed in \cite{bardi2014linear, feleqi2013derivation, fischer2017connection, biswas2015mean, arapostathis2017solutions, lacker2016general}.  The goal is to find sufficient conditions such that the limit of the sequences of Nash equilibrium for the $N$-player games identify as a solution of the corresponding mean-field game as $N \to \infty$. Convergence of Nash equilibria of symmetric $N$-player games to the corresponding mean-field games for stationary continuous-time problems with ergodic costs has been investigated in  \cite{bardi2014linear, feleqi2013derivation}. Moreover, such a convergence of Nash equilibria for symmetric $N$-player games to the corresponding mean-field solution for a broad class of continuous time symmetric games has been established in \cite{fischer2017connection} under uniform integrability and exchangeability (symmetry) conditions (see \cite[Theorem 5.1 and conditions (T) and (S)]{fischer2017connection}) provided that the cost function and dynamics admit the structural restrictions. 
%Using a concentration of measures argument, it has been shown in \cite{fischer2017connection} that, sequences of $\epsilon_{N}$- local (for each player) Nash equilibria for $N$ player games converge to a solution for the mean-field game under exchangeability of the initial states of the dynamics and weak convergence of normalized occupational measures to a deterministic measure \cite[Theorem 5.1]{fischer2017connection}. 
In \cite{lacker2016general}, assumptions on equilibrium policies of the large population mean-field symmetric stochastic differential games have been relaxed to allow the convergence of asymmetric approximate Nash equilibria to a weak solution of the mean-field game \cite[Theorem 2.6]{lacker2016general}. In a discrete-time setup, \cite{biswas2015mean} considered convergence of Nash equilibria for games with the mean-field interaction and with ergodic costs for Markov processes. The convergence result has been derived under an existence assumption on the mean-field solution and an additional convexity condition (see \cite[Theorem 5.1 and condition (A7)]{biswas2015mean}). In contrast, in the context of stochastic teams with countably infinite number of decision makers, the gap between person-by-person optimality (Nash equilibrium in the game-theoretic context) and global team optimality is significant since a perturbation of finitely many policies fails to deviate the value of the expected cost, thus person by person optimality is a weak condition for such a setup, and hence the results presented in the aforementioned papers may be inconclusive regarding global optimality of the limit equilibrium. This observation motivates us to investigate the connection between person-by-person-optimality and global team optimality in stochastic teams with countably infinite decision makers. Compared with \cite{bardi2014linear, feleqi2013derivation, fischer2017connection, biswas2015mean, arapostathis2017solutions} where only the convergence of a sequence of Nash equilibria for symmetric games with the mean-field interaction has been studied, we show that, under sufficient conditions, sequences of optimal policies for teams with $N$ number of decision makers as $N \to \infty$ converge to a team optimal policy for static teams with countably infinite number of decision makers. 

Related to mean-field team problems, a limit theory for \textit{mean-field type problems} (also called \textit{Mckean-Vlasov stochastic control problems}) has been established in \cite{lacker2017limit,carmona2018probabilistic}. In \cite{lacker2017limit, carmona2018probabilistic}, the connection between solutions of $N$-player differential control systems  and solutions of Mckean-Vlasov control problems has been investigated. It has been shown that the sequence of empirical measures of pairs of states and $\epsilon_{N}$-centralized optimal controls (under the classical information structure since all the information available are completely shared between players) converges in distribution as $N \to \infty$ to limit points in the set of pairs of states and optimal controls of the Mckean-Vlasov problem \cite{lacker2017limit} (see Remark \ref{rem:lacker}). In contrast, our focus is on the information structures of decision makers. Here, under convexity of the cost function and symmetry, we show the convergence of a sequence of decentralized optimal policies of $N$-DM teams  to an optimal policy of mean-field teams as $N \to \infty$. 
  
{\bf Connections with the literature on LQG games/teams.} There has been a number of studies focusing on the LQG setup (in addition to \cite{huang2012social, wang2017social}). A close study is \cite{mahajan2013static} where LQG static teams with countably infinite number of decision makers have been studied and sufficient conditions for global optimality have been established. In our paper, we utilize some of the results from \cite{mahajan2013static}, however compared with \cite{mahajan2013static}, we propose sufficient conditions for team optimality on average cost problems for a general setup: except convexity, no specific structure is presumed a priori on the cost function. For our analysis, we do not restrict the setup to the LQG one, where often direct methods can be applied building on \cite{Radner}, \cite{KraMar82}, and operator theory involving matrix algebra; in addition, we also study the mean-field setting. In fact, for a general setup of static teams, we introduce sufficient conditions (see Theorem \ref{the:5} and Theorem \ref{the:6}) such that the optimal cost and optimal policies of static teams with countably infinite number of decision makers is obtained as a limit of the optimal cost and optimal policies for static teams with $N$ number of decision makers as $N \to \infty$. In \cite{gattami2017team}, LQG team problems with infinitely many decision makers have been considered for a setup where the cost function is the expected inner-product of an infinite dimensional vector (and to allow for a Hilbert theoretic formulation, finiteness of the infinite sum of the moments of individual random variables is imposed) and linearity and uniqueness of optimal policies have been established; the finiteness (of the infinite summation) restriction rules out the setup in our paper. In \cite{ouyang2018optimal}, infinite horizon decentralized stochastic control problems containing a remote controller and a collection of local controllers dealing with linear models have been addressed for a setup where the cost is quadratic and the communication model satisfies a specified sharing pattern of information between local controller and remote controller. Under the assumed sharing pattern (common information), the connections between the optimal solution and the coupled algebraic Riccati equation for Markov jump linear systems and its convergence to the coupled fixed point equations have been utilized to show the optimality of the solution \cite{ouyang2018optimal}.
  
As a further motivation for our study, we note that for dynamic team problems, Ho and Chu \cite{HoChu} have introduced a technique such that dynamic partially nested LQG team problems can be reduced to static team problems (we also note that Witsenhausen \cite{WitsenStandard} showed that under an absolute continuity condition, any sequential dynamic team can be reduced to a static one). For infinite-horizon dynamic team problems, this reduction leads to a static team with countably many decision makers; thus leading to a different setup where our results in this paper will be applicable. We will study a particular example as a case study. In particular, the question of whether partially nested dynamic LQG teams admit optimal policies under an expected average cost criterion, in its most general form, has not been conclusively addressed despite the presence of results which impose linearity apriori for the optimal policies under such information structures \cite{Rotkowitz}. We hope that our solution approach can be utilized in the future to develop a complete theory for such problems. %We consider a special case as an example, as we review the contributions in the following. 

%Some of the results in this paper are presented in the conference paper \cite{sinacdc2018optimal}; this conference paper mainly announces the main results of the current paper and does not include the technically refined statements, the proofs of the results, except a sketch of Theorem \ref{the:4} and Theorem \ref{the:6}, and most of the examples are not included with details or not presented in the general formulation as presented here.  

% \subsection{Contributions} 
%The contribution of this paper is as follows:  
{\bf Contributions.} 
%Our paper has two main contributions.
\begin{itemize}
\item[(i)] For a general setup of static teams, we show that (see Theorem \ref{the:6}), under a uniform integrability condition (see Remark \ref{rem:3}), if sequences of team optimal policies of decision makers $i=1,\dots,N$ of static teams with $N$ number of decision makers converge uniformly in $i=1,\dots,N$ (see (b) in Theorem \ref{the:6}), then the corresponding limit policies are team optimal for the static team with countably infinite number of decision makers, under the expected average cost criteria. 
\item[(ii)] We establish global optimality results for mean-field teams under strict decentralization of the information structure for both teams with large numbers of players and infinitely many players. Toward this end, we introduce a notion of symmetrically optimal teams (see Definition \ref{Def:ost}) to obtain a global optimality result under relaxed sufficient conditions (see Section \ref{sec:4}). Under mild conditions on action spaces and observations of decision makers, through concentration of measures arguments, we establish the convergence of optimal policies for symmetric mean-field teams with $N$ decision makers to the corresponding optimal policy of mean-field teams (see Section \ref{sec:4}). In addition, we establish an existence result for optimal policies on mean-field teams under relaxed conditions on action spaces and the cost function (see Theorem \ref{the:existence}). 
\item[(iii)] We apply our results to a number of illustrative examples: We first consider LQG and LQ (non-Gaussian) average cost problems with state coupling (see Section \ref{Ex:1}  and Section \ref{Ex:2}). We also consider LQG average cost problems with  control coupling (see Section \ref{Ex:3}). In addition, we show that the team optimal policy of LQG teams with classical information structure (see Section \ref{Ex:4}) is obtained using the technique proposed in this paper. This is important since this result, while is well-known in the stochastic control literature, has not been investigated using static reduction proposed in \cite{HoChu} and hence this approach can be viewed as a step to address optimal solutions for infinite-horizon partially nested dynamic LQG problems which can be reduced to a static team with countably infinite number of decision makers.
\end{itemize}
The organization of the paper is as follows. Preliminaries and the problem statement are presented in Section \ref{sec:2}. Section \ref{sec:3} contains our main results including sufficient conditions for team optimality and asymptotic optimality for a general setup of static teams with countably infinite number of decision makers. Section \ref{sec:4} discusses symmetric and mean-field teams, and  applications are presented in Section \ref{sec:5}. Section \ref{sec:6} presents concluding remarks.

\section{Problem Formulation}\label{sec:2}
\subsection{Preliminaries}\label{sec:2.1}
%Optimal Solutions for Stochastic Teams with Countably Many Decision Makers}\label{sec:2}
Before presenting our main results, we introduce preliminaries following the presentation in \cite{YukselSaldiSICON17}, in particular, we introduce the characterizations laid out by Witsenhausen, through his {\it Intrinsic Model} \cite{wit75}; further characterizations and classifications of information structures are introduced comprehensively in \cite{YukselBasarBook}.
Suppose there is a pre-defined order in which the decision makers act. Such systems are called {\it sequential systems}.
The action and measurement spaces are standard Borel spaces, that is, Borel subsets of complete, separable and metric spaces. The {\it Intrinsic Model} for sequential teams is defined as follows.

\begin{itemize}
\item There exists a collection of {\it measurable spaces} \/$\{(\Omega, {\cal F}),
(\mathbb{U}^i,{\cal U}^i), (\mathbb{V}^i,{\cal V}^i), i \in {\mathcal{N}}\}$\@, specifying the system's distinguishable events, and control and measurement spaces, where $\mathcal{N}$ is either $\{1,\dots,N\}$ or $\mathbb{N}$ ($\mathbb{N}$ denotes the set of natural numbers). In this model (described in discrete time),  any action applied at any given time $t \in \mathcal{N}$ is regarded as applied by a decision maker DM$^{i}$ for $i \in \mathcal{N}$, who acts only once. The pair $(\Omega, {\cal F})$ is a
measurable space (on which an underlying probability may be defined). The pair $(\mathbb{U}^i, {\cal U}^i)$
denotes the measurable space from which the action, $u^i$, of decision maker $i$ is selected. The pair $(\mathbb{V}^i,{\cal V}^i)$ denotes the measurable observation/measurement space.

\item There is a \textit{measurement constraint} to establish the connection between the observation variables and the system's distinguishable events. The $\mathbb{V}^i$-valued observation variables are given by $v^i=h^i(\omega,{\underline u}^{[1,i-1]})$, where ${\underline u}^{[1,i-1]}=\{u^k, k \leq i-1\}$, $h^i$ are given measurable functions and $u^k$ denotes the action of DM$^k$. Hence, $v^i$ induces $\sigma(v^i)$ over $\Omega \times \prod_{k=1}^{i-1} \mathbb{U}^k$.
\item The set of admissible control laws $\underline{\gamma}= \{\gamma^1, \gamma^2, \dots\}$, also called
{\textit{designs}} or {\textit{policies}}, are measurable control functions, so that $u^i = \gamma^i(v^i)$. Let $\Gamma^i$ denote the set of all admissible policies for DM$^i$.
\end{itemize}
\begin{itemize}
\item There is a {\textit{probability measure}} $\mathbb{P}$ on $(\Omega, {\cal F})$ describing the probability space on which the system is defined.
\end{itemize}

Under this intrinsic model, a sequential team problem is {\textit{dynamic}} if the
information available to at least one DM is affected by the action of at least one other DM. A team problem is {\it static}, if for every decision maker the information available is only affected by exogenous disturbances; that is no other decision maker can affect the information of any given decision maker. \par
Information structures can also be categorized as {\it classical}, {\it quasi-classical} or {\it non-classical}. An Information Structure (IS) $\{v^i, i \in \mathcal{N} \}$ is {\it classical} if $v^i$ contains all of the information available to DM$^k$ for $k < i$. An IS is {\it quasi-classical} or {\it partially nested}, if whenever $u^k$, for some $k < i$, affects $v^i$ through the measurement function $h^i$, $v^i$ contains $v^k$ (that is $\sigma(v^k) \subset \sigma(v^i)$). An IS which is not partially nested is {\it nonclassical}.

\begin{itemize}
\item[\bf{($\mathcal{P}_{N}^{\prime}$)}]
Let $N=|{\cal N}|$ be the number of control actions taken, and each of these actions is taken by a different decision maker, where $\mathcal{N}:=\{1,\dots,N\}$. Let $\underline{\gamma}_{N} = \{\gamma^1, \cdots, \gamma^N\}$ and let ${\bf \Gamma}_{N} = \prod_{i}^{N} \Gamma^i$ be the space of admissible policies for the team with $N$-DMs. Assume an expected cost function is defined as
\begin{equation}\label{eq:1.1}
J_{N}(\underline{\gamma}_{N}) = E^{\underline{\gamma}_{N}}[c(\omega_{0},\underline{u}_{N})],
\end{equation}
for some Borel measurable cost function $c: \Omega_{0} \times \prod_{k=1}^{N} \mathbb{U}^k \to \mathbb{R}$ where $E^{\underline{\gamma}_{N}}[c(\omega_{0},\underline{u}_{N})] := E[c(\omega_{0},\gamma^1(v^1),\cdots,\gamma^N(v^N))]$ and  {we define $\omega_{0}$ as the cost function relevant exogenous random variable as $\omega_{0}:(\Omega,\mathcal{F}, \mathbb{P}) \to (\Omega_{0},{\cal B}(\Omega_0))$.} Here, we have the notation $\underline {u}_{N}:=\{u^i, i \in {\cal N}\}$ and $\mathcal{B}(\cdot)$ denotes the Borel $\sigma$-field.
\end{itemize}
\begin{definition}Team optimal solution for ($\mathcal{P}_{N}^{\prime}$) \cite{YukselBasarBook}.\\
For a given stochastic team problem with a given information
structure, a policy (strategy)\/ $N$\@-tuple\/ ${\underline \gamma}^{*}_{N}:=({\gamma^1}^*,\ldots, {\gamma^N}^*)\in {\bf \Gamma}_{N}$\@ is
 \textit{optimal} (\textit{team-optimal solution}) for ($\mathcal{P}_{N}^{\prime}$) if
\begin{equation*}
J_{N}({\underline \gamma}^*_{N})=\inf_{{{\underline \gamma}_{N}}\in {{\bf \Gamma}_{N}}}
J({{\underline \gamma}}_{N})=:J^*_{N}.
\end{equation*}  %The cost level achieved by this strategy,\/ $J^*_{N}$\@, is the \textit{
%minimum {\rm (or} optimal{\rm )} team cost}.
\end{definition}
%
%\begin{defn}\label{Def:TB2} \index{Person-by-person-optimality}
%For a given $N$-person stochastic team with a fixed information structure, $\{J;
%\Gamma^i, i \in {\cal N}\}$, an $N$-tuple of strategies
%${\underline \gamma}^*:=({\gamma^1}^*,\ldots, {\gamma^N}^*)$ constitutes a {\it Nash
%equilibrium} (synonymously,  a {\it person-by-person optimal} (pbp optimal) solution) if, for all $\beta
%\in \Gamma^i$ and all $i\in {\cal N}$, the following inequalities hold:
%\begin{equation}{J}^*:=J({\underline \gamma}^*) \leq J({\underline \gamma}^{-i}^*,
%\beta), \label{eq:7}\end{equation} where we have adopted the notation
%\begin{equation}({\underline \gamma}^{-i}^*,\beta):= ({\gamma^1}^*,\ldots, {\gamma^{i-1}}^*,
%\beta, {\gamma^{i+1}}^*,\ldots, {\gamma^N}^*). \label{eq:8}\end{equation}
%\end{defn}

\begin{definition}Person-by-person optimal solution \cite{YukselBasarBook}.\\
For a given\/ $N$\@-DM stochastic team with a fixed information structure, an\/ $N$\@-tuple of strategies\/ ${\underline \gamma}^*_{N}:=({\gamma^1}^*,\ldots, {\gamma^N}^*)$\@ constitutes a \textit{person-by-person optimal} (pbp optimal) solution for ($\mathcal{P}_{N}^{\prime}$) if, for all\/ $\beta
\in \Gamma^i$\@ and all\/ $i\in {\cal N}$\@, the following inequalities hold:
\begin{equation*}\label{eq:1.3}
{J}^*_{N}:=J_{N}({\underline \gamma}^*_{N}) \leq J_{N}({\underline \gamma}^{-i*}_{N},
\beta),
\end{equation*}
where
$({\underline \gamma}_{N}^{-i*},\beta):= (\gamma^{1*},\ldots, \gamma^{(i-1)*},
\beta, \gamma^{(i+1)*},\ldots, \gamma^{N*})$.
\end{definition}

\par
To simplify notations, let for any\/ $1 \leq k \leq N$\@,\/ $\underline{\gamma}^{-k}_{N} := \{\gamma^i, i \in \{1,\cdots,N\} \setminus \{k\} \}$\@.

\begin{definition}\label{def:sta}
Stationary solution \cite{Radner}.\\
A policy \/ $\underline\gamma_{N}(.)$\@ is stationary if\/ $J(\underline{\gamma}_{N})<\infty$\@, and for all\/ $i=1,...,N$\@,\/ $\mathbb{P}$\@-almost surely
\begin{equation*}\label{eq:1.5}
 \nabla_{u^{i}} \mathbb{E}\bigg[c(\omega_{0},(\underline\gamma ^ {-i}_{N},u^{i}))\bigg|v ^{i}\bigg] \bigg|_{u^{i}=\gamma^{i}(v^{i})}=0,
\end{equation*}
where $\nabla_{u^{i}}$ denotes the gradient with respect to $u^{i}$.
 \end{definition}
In this subsection, without abuse of notations, we sometimes used $\gamma^{i}$ as $\gamma^{i}(v^{i})$. In the following, we present some related existing results for static teams with $N$ decision makers.
%\begin{definition} Local finiteness \cite{Radner}.\\
%The cost functional\/ $J_{N}(\underline\gamma_{N})$\@ is locally finite at \/ $\underline\gamma_{N} \in \bf{\Gamma}_{N}$\@ if
%\begin{itemize}
%\item[(a)] $|J_{N}(\underline\gamma_{N})|<\infty$;
%\item[(b)] for any decision function $\underline\delta_{N} \in \bf{\Gamma}_{N}$ such that $|J(\underline\gamma_{N}+\underline\delta_{N})|<\infty$, there exist $k_{1},k_{2},...,k_{N}$ all positive such that 
%\begin{equation*}
%J_{N}(\gamma^{1}+h_{1}\delta^{1},....,\gamma^{N}+h_{N}\delta^{N})<\infty
%\end{equation*}
%for all\/ $h_{1},...,h_{N}$\@ for which 
%\begin{equation*}
%|h_{1}|\leq k_{1},...,|h_{N}|\leq k_{N}.
%\end{equation*}
%\end{itemize}
%\end{definition}
The following is known as Radner's theorem \cite{Radner}. Radner proposed the first result to connect the stationarity concept and global team optimality. %Note that Radner considered the static team.
\begin{theorem}\cite{Radner}\label{the:1}
If
\begin{itemize}
\item[(a)] $c(\omega_{0},\underline {u}_{N})$\@ is convex and differentiable in\/ $\underline u_{N}$ for $\mathbb{P}$\@- almost surely;
\item[(b)] $\inf\limits_{\underline\gamma_{N} \in \bf{\Gamma}_{N}}J_{N}(\underline\gamma_{N})>-\infty$\@;
\item[(c)] $J_{N}(.)$\@ is locally finite at\/ $\underline\gamma^{*}_{N}$\@  \cite{Radner};
\item[(d)] $\underline\gamma^{*}_{N}$\@ is stationary;
\end{itemize}
then\/ $\underline\gamma^{*}_{N}$\@ is globally optimal for ($\mathcal{P}_{N}^{\prime}$).
\end{theorem}
Radner's theorem fails in some applications because of the restrictive local finiteness assumption. Krainak et al \cite{KraMar82} relaxed assumptions and presented sufficient conditions for team optimality on static teams.
\begin{theorem}\cite{KraMar82}\label{the:2}
Assume that, for every fixed\/ $\omega_{0}$\@,\/ $c(\omega_{0},\underline u_{N})$\@ is convex differentiable in\/ $\underline u_{N}$\@.
Suppose (b) in Theorem \ref{the:1} holds.
Let\/ $\underline\gamma^{*}_{N} \in \bf{\Gamma}_{N}$\@, and assume that\/ $\mathbb{E}[c(\omega_{0},\underline\gamma^{*}_{N}(\underline{v}_{N}))]<\infty$\@. If, for all\/ $\underline\gamma_{N} \in \bf{\Gamma}_{N}$\@ with\/ $\mathbb{E}[c(\omega_{0},\underline\gamma_{N}(\underline v_{N}))]<\infty$\@,
\begin{equation}\label{eq:1.7}
\mathbb{E}\bigg[\sum_{i=1}^{N}c_{u^{i}}(\omega_{0},\underline\gamma^{*}_{N})(\gamma^{i}-\gamma^{i*})\bigg] \geq 0,
\end{equation}
where\/ $c_{u^{i}}(\omega_{0},\underline\gamma^{*}_{N})$\@ is the partial derivative of\/ $c(\omega_{0},\underline u_{N})$\@ with respect to\/ $u^{i}$\@ valued in\/ $\underline u_{N}=\underline \gamma^{*}_{N}$\@, then\/ $\underline\gamma^{*}_{N}$\@ is an optimal team policy for ($\mathcal{P}_{N}^{\prime}$). Moreover, if\/ $c(\omega_{0},\underline u_{N})$\@ is strictly convex in\/ $\underline u_{N}$\@ \/ $\mathbb{P}$-almost surely, then\/ $\underline\gamma^{*}_{N}$\@ is \/ $\mathbb{P}$\@-a.s. unique.
\end{theorem}
Since the set of admissible policies is generally uncountable, checking \eqref{eq:1.7} is difficult. Krainak et al \cite{KraMar82} further developed relaxed conditions under which stationarity of a policy implies its optimality. %Radner used local finiteness to establish such a result, but local finiteness is restrictive.
\begin{theorem}\cite{KraMar82}\label{the:3}
Assume that, for every fixed\/ $\omega_{0} \in \Omega_{0}$\@,\/ $c(\omega_{0},\underline u_{N})$\@ is a convex differentiable function of\/ $\underline u_{N}$\@ and suppose (b) in Theorem \ref{the:1} holds. Assume that\/ $\underline\gamma^{*}_{N} \in \bf{\Gamma}_{N}$\@ is a stationary policy. Let, for all\/ $\underline\gamma_{N} \in \bf{\Gamma}_{N}$\@ with\/ $\mathbb{E}[c(\omega_{0},\underline\gamma_{N}(\underline v_{N}))]<\infty$\@,
\begin{equation}\label{eq:1.1.8}
\mathbb{E}\bigg[c_{u^{i}}(\omega_{0},\underline\gamma^{*}_{N})(\gamma^{i}-\gamma^{i*})\bigg]<\infty~ \text{for}~i=1,...,N.
\end{equation}
Then\/ $\underline\gamma^{*}_{N}$\@ is a team optimal policy for ($\mathcal{P}_{N}^{\prime}$). If\/ $c(\omega_{0},\underline u_{N})$\@ is strictly convex in\/ $\underline u_{N}$\@,\/ $\mathbb{P}$\@-a.s., then\/ $\underline \gamma^{*}_{N}$\@ is unique.
\end{theorem}
Furthermore, \eqref{eq:1.1.8} can be replaced by the following more checkable conditions \cite{YukselBasarBook}: Let $\Gamma^{i}$ be Hilbert space for each $i=1,...,N$ and $\mathbb{E}[c(\omega_{0},\underline\gamma_{N}(\underline v_{N}))]<\infty$\@ for all $\underline\gamma_{N} \in \bf{\Gamma}_{N}$. Moreover, let
\begin{equation}
\mathbb{E}\bigg[c_{u^{i}}(\omega_{0},\underline\gamma^{*}_{N})\bigg|v^{i}\bigg]\in \Gamma^{i},~i=1,...,N.
\end{equation}
The above conditions follows directly from \eqref{eq:1.1.8} when $\Gamma^{i}$ is a Hilbert space for all $i=1,2,\dots,N$. This condition can be checked for some applications; for example, LQ teams \cite{YukselBasarBook}.

\subsection{Problem statement}\label{sec:2.2}

\begin{itemize}
\item[\bf{($\mathcal{P}_{\infty}$)}]
Consider a team with countably infinitely many decision makers. Let $\bf{\Gamma}=\prod_{i \in \mathbb{N}} \Gamma^{i}$ be a countable but an infinite product policy space. We assume $\mathbb{U}^{i}=\mathbb{R}^{n}$, and $\mathbb{V}^{i}=\mathbb{R}^{m}$ for all $i \in \mathbb{N}$, where $n$ and $m$ are positive integers. Let $c:\Omega_{0} \times \mathbb{R}^{n}\times \mathbb{R}^{n} \rightarrow \mathbb{R}_{+}$, and the expected cost be
\begin{equation}\label{eq:2.5.5}
J(\underline{\gamma})=\limsup\limits_{N\rightarrow \infty} \frac{1}{N} \mathbb{E}^{\underline{\gamma}}\bigg[\sum_{i=1}^{N}c(\omega_{0},u^{i},\frac{1}{N}\sum_{p=1}^{{N}}u^{p})\bigg],
\end{equation}
where we denote 
$\mathbb{E}^{\underline{\gamma}}[\sum_{i=1}^{N}c(\omega_{0},u^{i},\frac{1}{N}\sum_{p=1}^{{N}}u^{p})]:=\mathbb{E}\left[\sum_{i=1}^{N}c(\omega_{0},\gamma^{i}(v^{i}),\frac{1}{N}\sum_{p=1}^{{N}}\gamma^{p}(v^{p}))\right]$.
\end{itemize}
\begin{definition}Team optimal solution for ($\mathcal{P}_{\infty}$).\\
For a given stochastic team problem with a given information
structure, a policy\/ ${\underline \gamma}^{*}:=({\gamma^{1*}},\gamma^{2*},\ldots)\in {\bf \Gamma}$\@ is
 \textit{optimal} for ($\mathcal{P}_{\infty}$) if
\begin{equation*}
J({\underline \gamma}^*)=\inf_{{{\underline \gamma}}\in {{\bf \Gamma}}}
J({{\underline \gamma}})=:J^* .
\end{equation*}   %The cost level achieved by this policy,\/ $J^*$\@, is the \textit{optimal team cost}.
\end{definition}
Our goal in this paper is to establish conditions for a team policy to be optimal, and also connect the optimal cost and policies for {($\mathcal{P}_{\infty}$)} and {($\mathcal{P}_{N}$)}. To this end, we re-define {($\mathcal{P}_{N}$)} for our problem statement as follows:
\begin{itemize}
\item[\bf{($\mathcal{P}_{N}$)}]
Let $N=|{\cal N}|$ be the number of control actions taken and $\underline{\gamma}_{N} = \{\gamma^1, \cdots, \gamma^N\}$ and let ${\bf \Gamma}_{N} = \prod_{i}^{N} \Gamma^i$ space of admissible policies for the team with $N$-DMs. Assume an expected cost function is defined as
\begin{equation}\label{eq:2.6.6}
J_{N}(\underline{\gamma}_{N}) = \frac{1}{N} \mathbb{E}^{\underline{\gamma}_{N}}\bigg[\sum_{i=1}^{N}c(\omega_{0},u^{i},\frac{1}{N}\sum_{p=1}^{{N}}u^{p})\bigg].
\end{equation}
\end{itemize}
We will investigate the relations between the sequence of solutions to \eqref{eq:2.6.6} and the solution to \eqref{eq:2.5.5}. We note that our main result is on the connection between {($\mathcal{P}_{\infty}$)} and {($\mathcal{P}_{N}$)}.
%\sy{Sina; please address the following two items:}
%\sy{In particular, we will establish relations between solutions to $\mathcal{P}_{N}$ and $\mathcal{P}_{\infty}$ where $\mathcal{P}_{\infty}$ is viewed to be a limit problem for $\mathcal{P}_{\infty}$ with $\mathcal{P}_{N}$ defined as with the cost being $\mathbb{E}^{\underline{\gamma}}(\sum_{i=1}^{N}c(\omega,u^{i},\frac{1}{N}\sum_{p=1}^{{N}}u^{p})):=\mathbb{E}\left[\sum_{i=1}^{N}c(\omega,\gamma^{i}(v^{i}),\frac{1}{N}\sum_{p=1}^{{N}}\gamma^{p}(v^{p}))\right]$. PLEASE REVISE. As it is, Definition 2.1 is for a general team problem and not the problem whose limit defines $P_{\infty}$.}
%\sy{The relation between $\mathcal{P}_{\infty}$ and $\mathcal{P}_{N}$ is not clear in Section 4. For example in the statement of Theorem 4.9, we start with  a fixed $\mathcal{P}_{N}$ but then we don't discuss the relation; the existence statement is vague; the conditions should concern the setup here but we should relate $\mathcal{P}_{\infty}$ and $\mathcal{P}_{N}$. Perhaps here or at the beginning of Section 2 this relation should be made explicit. Perhaps re-define $\mathcal{P}_{N}$ in that section with the normalized cost perhaps? I think this will resolve the issue. }

\section{Optimal policies for teams with infinitely many decision makers}\label{sec:3}

\subsection{Sufficient conditions of optimality}
In the following, we propose sufficient conditions of team optimality for $(\mathcal{P}_{\infty})$. We often follow \cite{KraMar82}, and the result is an extension of  \cite{KraMar82} to a general setup of static teams with countably infinite number of decision makers. We also note a related analysis in \cite{mahajan2013static}. We will use the following theorem for LQ static teams with countably infinite number of decision makers (see Section \ref{Ex:2}).
\begin{assumption}
Let 
\begin{itemize}
\item[{(A1)}] $c(\omega_{0},u^{i},\frac{1}{N}\sum_{p=1}^{{N}}u^{p})$\@ be a $\mathbb{R}_{+}$-valued jointly convex function of second and third arguments and differentiable in\/ $u^{i}$\@ with continuous partial derivatives, for every\/ $\omega_{0} \in \Omega_{0}$\@.
\item[(A2)] for some\/ $\underline{\gamma}^{*} \in \bf{\Gamma}$\@,
\begin{flalign}\label{eq:2.1}
\lim\limits_{N\rightarrow \infty}& \frac{1}{N}\sum_{i=1}^{N} \mathbb{E}^{\underline{\gamma}^{*}}\bigg[c(\omega_{0},u^{i},\frac{1}{N}\sum_{p=1}^{{N}}u^{p})\bigg]<\infty.
\end{flalign}
\end{itemize}
\end{assumption}
We note that the cost function is differentiable in $u^{i}$ which means that the cost is totally differentiable in $u^{i}$, i.e.,  $\frac{d}{du^{i}}c(\omega_{0},u^{i},\frac{1}{N}\sum_{p=1}^{N}u^{p})=\frac{\partial}{\partial u^{i}}c(\omega_{0},u^{i},\mu_{N})+\frac{1}{N}\frac{\partial}{\partial\mu_{N}}c(\omega_{0},u^{i},\mu_{N})$.
\begin{theorem}\label{the:4}
 Assume (A1) holds and (A2) holds for $\underline\gamma^{*}\in\bf{\Gamma}$. If for all\/ $\underline\gamma \in \bf{\Gamma}$\@ with\/ $J(\underline{\gamma}) < \infty$\@,
\begin{equation}\label{eq:2.2}
\limsup\limits_{N\rightarrow \infty} \frac{1}{N}\mathbb{E}\bigg[\sum_{i=1}^{N}\sum_{k=1}^{N}c_{u^{k}}(\omega_{0},\gamma^{i*},\mu^{*}_{N})(\gamma^{k}-\gamma^{k*})\bigg] \geq 0,
\end{equation}
where $\mu^{*}_{N}=\frac{1}{N}\sum_{p=1}^{N}\gamma^{p*}(v^{p})$, then\/ $\underline\gamma^{*}$\@ is a globally optimal team policy for ($\mathcal{P}_{\infty}$).
\end{theorem}
\begin{proof}
Under (A1), the required derivatives in \eqref{eq:2.2} in the direction of $u^{i}$ exist and the chain rule of derivatives can be applied  since this implies that the cost function is Fr\'echet differentiable in $u^{i}$ \cite{Fleming}. Now, we use the convexity property to justify interchanging the expectation and the derivation similar to {\cite[Theorem 2]{KraMar82}}, then we use \eqref{eq:2.1} and \eqref{eq:2.2} to establish the global optimality of $\underline\gamma^{*}$ for ($\mathcal{P}_{\infty}$). 
%The steps of the proof are similar to what presented in \cite{mahajan2013static}, where LQG teams were considered; here the structure of the cost function is not restricted. 
Under (A1), we have for every $\alpha \in (0,1]$,
\begin{flalign*}
\sum_{i=1}^{N}c(\omega_{0},& \gamma^{i*}+\alpha \delta^{i},\mu^{*}_{N}+ \frac{\alpha}{N}\sum_{p=1}^{{N}} \delta^{p}) - c(\omega_{0}, \gamma^{i*},\mu^{*}_{N})\\&\leq\alpha \sum_{i=1}^{N}\left(c(\omega_{0}, \gamma^{i},\mu_{N})-c(\omega_{0}, \gamma^{i*},\mu^{*}_{N})\right),
\end{flalign*}
where  $\mu_{N}=\frac{1}{N}\sum_{p=1}^{N}\gamma^{p}(v^{p})$ and $\delta^{i}=\gamma^{i}-\gamma^{i*}$. Let
\begin{flalign*}
h^{\omega_{0}}_{N}(\alpha):=\frac{1}{\alpha}\bigg[\frac{1}{N}&\sum_{i=1}^{N}c(\omega_{0},\gamma^{i*}+\alpha \delta^{i},\mu^{*}_{N}+\frac{\alpha}{N}\sum_{p=1}^{N}\delta^{p})\\& -c(\omega_{0}, \gamma^{i*},\mu^{*}_{N})\bigg].
\end{flalign*}
Hence, \cite[Proposition 6.3.2]{Dud02} implies that $h^{\omega_{0}}_{N}(\alpha)$ is a monotone non-increasing function as $\alpha \to 0$ in $\alpha \in[0, 1]$ and bounded from above by $h^{\omega_{0}}_{N}(1)$. Thus, by~\cite[Corollary 6.3.3]{Dud02}, $h_{+,N}^{\prime}(\omega_{0},0):=\lim_{\alpha \to 0}h^{\omega_{0}}_{N}(\alpha)$ exists. Since $h^{\omega_{0}}_{N}(\alpha)$ is a monotonic non-increasing function as $\alpha \to 0$ in $\alpha \in[0, 1]$ and bounded above by $h^{\omega_{0}}_{N}(1)$, and since $J(\underline{\gamma}^{*})$ and $J(\underline{\gamma})$ are finite, we can choose $N$ large enough such that $\mathbb{E}(h_{N}^{\omega_{0}}(1))<\infty$.
Now, we can use the monotone convergence theorem (see \cite[page. 170]{HernandezLermaMCP}) to interchange the limit and the expectation
\begin{equation}\label{eq:2.4}
\lim_{\alpha \to 0}\mathbb{E}(h^{\omega_{0}}_{N}(\alpha))=\mathbb{E}(\lim_{\alpha \to 0}h^{\omega_{0}}_{N}(\alpha))=\mathbb{E}(h_{+,N}^{\prime}(\omega_{0},0)).
\end{equation} 
%\begin{equation*}
 From \cite[Lemma 1]{KraMar82}, we have $\mathbb{E}(h_{+,N}^{\prime}(\omega_{0},0))=\frac{1}{N}\mathbb{E}(\sum_{i=1}^{N}\sum_{k=1}^{N}c_{u^{k}}(\omega_{0},\gamma^{i*},\mu^{*}_{N})\delta^{k})$.
%\end{equation*}
Define 
\begin{equation*}
F_{\underline\gamma_{N}}^{N}(\alpha):=\displaystyle\frac{1}{N}\mathbb{E}\bigg(\sum_{i=1}^{N}c(\omega_{0}, \gamma^{i*}+\alpha \delta^{i},\mu^{*}_{N}+\frac{\alpha}{N}\sum_{p=1}^{N} \delta^{p})\bigg).
\end{equation*}
Note that  $F_{\underline\gamma_{N}}^{N}(\alpha)$ exists for $\alpha \in [0,1]$ since $\mathbb{E}(h_{N}^{\omega_{0}}(\alpha)) \leq \mathbb{E}(h_{N}^{\omega_{0}}(1))<\infty$, and $\mathbb{E}(\frac{1}{N}\sum_{i=1}^{N}c(\omega_{0},\gamma^{i*},\mu^{*}_{N}))<\infty$.
Therefore, one can write $F ^{\prime N}_{\underline\gamma_{N}^{+}}(0)=\displaystyle\lim_{\alpha \to 0}\mathbb{E}(h^{\omega_{0}}(\alpha))$, and
\begin{flalign*}
F ^{\prime N}_{\underline\gamma_{N}^{+}} (0)=&\frac{1}{N}\mathbb{E}\bigg(\sum_{i=1}^{N}\sum_{k=1}^{N}c_{u^{k}}(\omega_{0},\gamma^{i*},\mu^{*}_{N})(\gamma^{k}-\gamma^{k*})\bigg).
\end{flalign*}
Thus, we can write
\begin{flalign}
 J(\underline{\gamma})-J(\underline{\gamma}^{*})&\ =\limsup\limits_{N\rightarrow \infty} F_{\underline\gamma_{N}}^{N}(1)-\limsup\limits_{N\rightarrow \infty}  F_{\underline\gamma_{N}}^{N}(0)\label{eq:2.5}\\
 &\ =\limsup\limits_{N\rightarrow \infty} F_{\underline\gamma_{N}}^{N}(1)-\liminf\limits_{N\rightarrow \infty}  F_{\underline\gamma_{N}}^{N}(0)\label{eq:2.6}\\
&\ \geq \limsup\limits_{N\rightarrow \infty} \frac{ F_{\underline\gamma_{N}}^{N}(1)-F_{\underline\gamma_{N}}^{N}(0)}{1}\label{eq:2.10}\\
&\ \geq \limsup\limits_{N\rightarrow \infty} F ^{\prime N}_{\underline\gamma_{N}^{+}} (0)\geq 0\label{eq:2.11},
\end{flalign}
where \eqref{eq:2.6} follows from (A2) and \eqref{eq:2.1}, and $-\liminf\limits_{N\rightarrow \infty}a_{N}=\limsup\limits_{N\rightarrow \infty} -a_{N}$, $\limsup\limits_{N\rightarrow \infty}a_{N}+\limsup\limits_{N\rightarrow \infty}b_{N} \geq \limsup\limits_{N\rightarrow \infty} (a_{N}+b_{N})$ imply \eqref{eq:2.10}, and \eqref{eq:2.11} holds since $F_{\underline\gamma_{N}}^{N}(.)$ is a convex function using \cite[Corollary 6.3.3]{Dud02}, and since $a_{N} \geq b_{N}$ then $\limsup\limits_{N\rightarrow \infty}a_{N} \geq \limsup\limits_{N\rightarrow \infty}b_{N}$. Finally, the last inequality follows from \eqref{eq:2.2}; hence, $J(\underline\gamma)-J(\underline\gamma^{*}) \geq 0$, and the proof is completed.
\end{proof}
In some applications, \eqref{eq:2.2} can be difficult to check since it must be satisfied for all\/ $\underline\gamma \in \bf{\Gamma}$\@ with\/ $J(\underline{\gamma}) < \infty$\@. In the next section, we address this issue by introducing a constructive approach for static teams with countably infinite number of decision makers as a limit of a sequence of team optimal policies of the corresponding static teams with finite number of decision makers. 
In the following, we propose sufficient conditions to approximate the optimal cost and a team optimal policy for static teams with countably infinite number of decision makers using the optimal cost and an optimal policy for static teams with $N$ decision makers.  {We note that our first result here is based on \cite[Theorem 1]{mahajan2013static}, which considered an equality. We denote $\underline\gamma|_{N} \in \bf{\Gamma_{N}}$ as a restriction of $\underline\gamma\in \bf{\Gamma}$ to the first $N$ components}.
%\begin{definition} (Extension of $\underline\gamma^{*}_{N}$)\\
%An extension of\/ $\underline\gamma^{*}_{N}$\@, the globally team optimal policy for a team with finitely many decision makers, is\/ $\underline\gamma^{*}_{N, \infty}=\displaystyle[\gamma_{N}^{1*},\gamma_{N}^{2*},...,\gamma_{N}^{N*},\beta,\beta,...] \in \bf\Gamma$\@, where\/ $\beta$ is an arbitrary policy.
%\end{definition}

%\sy{Do you use this definition here next? It seems you don't. We used this definition for (31) in Theorem 4.3}.

%The following theorem is based on the idea given from \cite{mahajan2013static}. The result in \cite{mahajan2013static} has been applied to the LQG team with countably many decision makers; however, here,  we consider a more general setup.

\begin{theorem}\label{the:5}
Let \/ $\underline\gamma^{*}_{N}\in \bf{\Gamma}_{N}$\@ be an optimal policy for ($\mathcal{P}_{N}$) as \eqref{eq:2.6.6} (see \cite{KraMar82, gupta2014existence, YukselSaldiSICON17} for sufficient conditions). If there exists\/ $\underline\gamma^{*} \in \bf\Gamma$\@, with\/ $J(\underline{\gamma}^{*}) < \infty$\@, satisfying
\begin{equation}\label{eq:2.14}
\limsup\limits_{N\rightarrow \infty}J_{N}(\underline\gamma^{*}_{N})\geq J(\underline\gamma^{*}),
\end{equation}
then $\underline\gamma^{*}$ is a globally team optimal policy for ($\mathcal{P}_{\infty}$).
\end{theorem}
\begin{proof}
We have
\begin{flalign}
J(\underline\gamma^{*})&\leq\limsup\limits_{N\rightarrow \infty} \frac{1}{N} \sum_{i=1}^{N} \mathbb{E}^{\underline{\gamma}^{*}_{N}}\left(c(\omega_{0},u^{i},\mu_{N})\right)\label{eq:1.2.16}\\
&=\limsup\limits_{N\rightarrow \infty} \inf_{\underline\gamma_{N} \in \bf{\Gamma}_{N}}\frac{1}{N} \sum_{i=1}^{N} \mathbb{E}^{\underline{\gamma}_{N}}\left(c(\omega_{0},u^{i},\mu_{N})\right)\label{eq:1.2.17}\\
&=\limsup\limits_{N\rightarrow \infty} \inf_{\underline\gamma \in \bf{\Gamma}}\frac{1}{N} \sum_{i=1}^{N} \mathbb{E}^{\underline{\gamma}}\left(c(\omega_{0},u^{i},\mu_{N})\right)\label{eq:1.2.18}\\
&\leq\inf_{\underline\gamma \in \bf{\Gamma}}\limsup\limits_{N\rightarrow \infty} \frac{1}{N} \sum_{i=1}^{N} \mathbb{E}^{\underline{\gamma}}\left(c(\omega_{0},u^{i},\mu_{N})\right)\label{eq:3.11}\\
&=\inf_{\underline\gamma \in \bf{\Gamma}}J(\underline\gamma)\nonumber,
\end{flalign}
where $\mu_{N}:=\frac{1}{N}\sum_{p=1}^{N}u^{p}$ and \eqref{eq:1.2.16} follows from \eqref{eq:2.14}, and \eqref{eq:1.2.17} is true since $\gamma^{*}_{N}$ is a team optimal policy for ($\mathcal{P}_{N}$) (see \eqref{eq:2.6.6}). Furthermore, \eqref{eq:1.2.18} follows from the fact that $[\underline\gamma|_{N}:\underline\gamma \in \bf\Gamma]=\bf\Gamma_{N}$, where $\underline\gamma|_{N}$ is $\underline\gamma$ restricted to the first $N$ components.
\end{proof}
\begin{remark}\label{lem:1.1.1}
Under (A2), one can replace \eqref{eq:2.14} with
\begin{equation}\label{eq:1.1.17}
\scalemath{0.92}{\limsup\limits_{N\rightarrow \infty}  \frac{1}{N} \sum_{i=1}^{N}\bigg[\mathbb{E}^{\underline{\gamma}^{*}_{N}}\bigg(c(\omega_{0},u^{i},\mu_{N})\bigg)-\mathbb{E}^{\underline{\gamma}^{*}}\bigg(c(\omega_{0},u^{i},\mu_{N})\bigg)\bigg]\geq 0}.
\end{equation}
\end{remark}
The above theorem and remark will be useful for some applications (see for example Section \ref{Ex:4}).
%\begin{remark}
%One can also define the extension periodically,\/ $\gamma^{i*}_{N, \infty}=\gamma^{i \Mod{N}*}_{N}$\@, i.e,  
%\begin{equation*}
%\underline\gamma^{*}_{N, \infty}=[\gamma_{N}^{1*},\gamma_{N}^{2*},...,\gamma_{N}^{N*},\gamma_{N}^{1*},\gamma_{N}^{2*},...,\gamma_{N}^{N*},...].
%\end{equation*}
%\end{remark}
\subsection{Asymptotically optimal policies as a limit of finite team optimal policies}
In the following, we present a sufficient condition for \eqref{eq:2.14}. The following result also presents a constructive method to obtain optimal policies using asymptotic analysis.
\begin{theorem}\label{the:6}
Assume
 \begin{itemize}
\item[(a)] {for every N}, there exist\/ $\underline{\gamma}^{*}_{N} \in \bf{\Gamma}_{N}$\@ for ($\mathcal{P}_{N}$) (see \eqref{eq:2.6.6}),
%\item[(b)] there exists\/ $\underline\gamma^{*}_{\infty}$\@ with\/ $J(\underline{\gamma}_{\infty}^{*}) < \infty$\@, satisfying \eqref{eq:2.1},
\item[(b)] let $\omega \in B$ for some $B \in \mathcal{F}$ event of $\mathbb{P}$ measure one, for every fixed $v^{i}(\omega)$, \/ ${\gamma}^{i*}_{N}(v^{i})$\@ converges to\/ ${\gamma}^{i*}_{\infty}(v^{i})$\@ uniformly in\/ $i=1,2,\dots,N$\@, i.e., 
\begin{equation*}
\lim\limits_{N\rightarrow \infty}\sup\limits_{1\leq i \leq N}|\gamma^{i*}_{N}(v^{i})-\gamma^{i*}_{\infty}(v^{i})|=0~\mathbb{P}-a.s.,
\end{equation*}
\item[(c)] 
there exists a\/ $\mathbb{P}$\@-integrable function\/ $g(\omega_{0},\underline{v})$\@ such that, for every $N$,
%\begin{equation*}
%\frac{1}{N}\sum_{i=1}^{N}c(\omega,\gamma^{i*}_{N}(v^{i}),\mu^{*}_{N}) \leq  g(w),
%\end{equation*}
%and
\begin{equation*}
\frac{1}{N}\sum_{i=1}^{N}c\bigg(\omega_{0},\gamma^{i*}_{\infty}(v^{i}),\frac{1}{N}\sum_{p=1}^{N}\gamma^{p*}_{\infty}(v^{p})\bigg) \leq  {g}(\omega_{0},\underline{v}),
\end{equation*}
\end{itemize}
 where $\underline{v}=(v^{1},v^{2},\dots)$, then\/ $\underline\gamma^{*}$\@, a team optimal policy for ($\mathcal{P}_{\infty}$), is a pointwise limit of\/ $\underline\gamma^{*}_{N}$\@, an optimal policy for ($\mathcal{P}_{N}$), i.e.,\/ ${\gamma}^{i*}(v^{i})=\displaystyle\lim\limits_{N\rightarrow \infty}{\gamma}^{i*}_{N}(v^{i})={\gamma}^{i*}_{\infty}(v^{i})$\@ $\mathbb{P}$-almost surely.
\end{theorem}
\begin{proof}
According to Theorem \ref{the:5}, we only need to show that 
\begin{flalign*}
\limsup\limits_{N\rightarrow \infty} J_{N}(\underline{\gamma}^{*}_{N})&\geq \liminf\limits_{N\rightarrow \infty} J_{N}(\underline{\gamma}^{*}_{N})\\&\geq E\bigg(\lim\limits_{N\rightarrow \infty} \frac{1}{N}\sum_{i=1}^{N}c(\omega_{0},{\gamma}^{i*}_{N}(v^{i}),\mu_{N}^{*})\bigg)\\&=\lim\limits_{N\rightarrow \infty} J_{N}(\underline{\gamma}^{*}_{\infty}),
\end{flalign*}
where $\mu^{*}_{N}=\frac{1}{N}\sum_{p=1}^{N}\gamma^{p*}_{N}(v^{p})$ and the second inequality follows from Fatou's lemma (since the cost function is non-negative). In the following, we justify the equality above. On a set of $\mathbb{P}$ measure one, $\omega \in B$ where $B \in \mathcal{F}$, for every fixed $v^{i}(\omega)$ in this set, define $\underline{v}(\omega)=(v^{1}(\omega),v^{2}(\omega),...)$ and $\underline{v}_{N}(\omega)=(v^{1}(\omega),\dots,v^{N}(\omega))$. We follow three steps to prove the theorem.

\begin{itemize}[wide = 0pt]
\item [\textbf{(Step 1):}]  We show that on a set of $\mathbb{P}$ measure one, $\omega \in B$ where $B \in \mathcal{F}$, for every fixed $v^{i}(\omega)$ in this set $\lim\limits_{N\rightarrow \infty}\frac{1}{N}\sum_{i=1}^{N}\left(\gamma^{i*}_{N}(v^{i})-\gamma^{i*}_{\infty}(v^{i})\right)=0$. For a fixed $\underline{v}$, following from (b) for a given $\delta_{\omega,\underline{v}_{N}}:=\sup_{1 \leq i \leq N}|\gamma^{i*}_{N}(v^{i})-\gamma^{i*}_{\infty}(v^{i})|>0$ there exists $\hat{N}(\delta_{\omega,\underline{v}_{N}}) \in \mathbb{N}$ such that for $N>\hat{N}(\delta_{\omega,\underline{v}_{N}})$, $|\gamma^{i*}_{N}(v^{i})-\gamma^{i*}_{\infty}(v^{i})|\leq\delta_{\omega,\underline{v}_{N}}$ for every $i=1,\dots,N$, where $\lim\limits_{N\rightarrow \infty}\delta_{\omega,\underline{v}_{N}}=0$ $\mathbb{P}$-almost surely. We have $\mathbb{P}$-almost surely,
 \begin{flalign*}
 &\bigg|\frac{1}{N}\sum_{i=1}^{N}\left(\gamma^{i*}_{N}(v^{i})-\gamma^{i*}_{\infty}(v^{i})\right)\bigg|<\frac{1}{N}\sum_{i=1}^{N} \delta_{\omega,\underline{v}_{N}}=\delta_{\omega,\underline{v}_{N}},
 \end{flalign*}
and since $\lim\limits_{N\rightarrow \infty}\sup_{1 \leq i \leq N}|\gamma^{i*}_{N}(v^{i})-\gamma^{i*}_{\infty}(v^{i})|= 0$, we have $\lim\limits_{N\rightarrow \infty}\delta_{\omega,\underline{v}_{N}}=0$. Hence, we can show that $\lim\limits_{N\rightarrow \infty}\frac{1}{N}\sum_{i=1}^{N}\gamma^{i*}_{N}(v^{i})=\lim\limits_{N\rightarrow \infty}\frac{1}{N}\sum_{i=1}^{N}\gamma^{i*}_{\infty}(v^{i})$. Following from continuity, $c(\omega_{0},\gamma^{i*}_{N}(v^{i}),\mu^{*}_{N})$ converges to $c(\omega_{0},\gamma^{i*}_{\infty}(v^{i}),\lim\limits_{N\rightarrow \infty}\mu_{\infty}^{*})$ $\mathbb{P}$-a.s. for every $i=1.\dots,N$.

\item [\textbf{(Step 2):}] 
We show that $c(\omega_{0},\gamma^{i*}_{N}(v^{i}),\mu^{*}_{N})$ converges to $c(\omega_{0},\gamma^{i*}_{\infty}(v^{i}),\lim\limits_{N\rightarrow \infty}\mu^{*}_{\infty})$ uniformly in $i=1,\dots,N$ $\mathbb{P}$-almost surely, where $\mu^{*}_{\infty}=\frac{1}{N}\sum_{p=1}^{N}\gamma^{p*}_{\infty}(v^{p})$. By continuity of the cost function, we have for a given $\epsilon_{\omega,{\underline{v}_{N}}}>0$, there exists $\delta_{\omega,{\underline{v}_{N}}}>0$ such that $|\gamma^{i*}_{N}(v^{i})-\gamma^{i*}_{\infty}(v^{i})|<\delta_{\omega,{\underline{v}_{N}}}$, and  $|\frac{1}{N}\sum_{i=1}^{N}\left(\gamma^{i*}_{N}(v^{i})-\gamma^{i*}_{\infty}(v^{i})\right)|<\delta_{\omega,\underline{v}_{N}}$ implies $|c(\omega_{0},\gamma^{i*}_{N}(v^{i}),\mu^{*}_{N})-c(\omega_{0},\gamma^{i*}_{\infty}(v^{i}),\mu^{*}_{\infty})|<\epsilon_{\omega,{\underline{v}_{N}}}$ $\mathbb{P}$-almost surely for every $i=1,\dots,N$. Following from (Step 1), we have for $N>\hat{N}(\delta_{\omega,\underline{v}_{N}}(\epsilon_{\omega,{\underline{v}_{N}}}))$, $|\gamma^{i*}_{N}(v^{i})-\gamma^{i*}_{\infty}(v^{i})|<\delta_{\omega,\underline{v}_{N}}$, and  $|\frac{1}{N}\sum_{i=1}^{N}\left(\gamma^{i*}_{N}(v^{i})-\gamma^{i*}_{\infty}(v^{i})\right)|<\delta_{\omega,\underline{v}_{N}}$. Hence, $\mathbb{P}$-a.s.
\begin{flalign*}
|c(\omega_{0},\gamma^{i*}_{N}&(v^{i}),\mu^{*}_{N})-c(\omega_{0},\gamma^{i*}_{\infty}(v^{i}),\mu^{*}_{\infty})| <\epsilon_{\omega,\underline{v}_{N}},
\end{flalign*}
where $\lim\limits_{N\rightarrow \infty}\epsilon_{\omega,\underline{v}_{N}}=0$.

\item [\textbf{(Step 3):}] 
In this step, we show that $\mathbb{P}$-a.s., 
\begin{equation*}
\lim\limits_{N\rightarrow \infty}\frac{1}{N}\sum_{i=1}^{N}\left(c(\omega_{0},\gamma^{i*}_{N}(v^{i}),\mu^{*}_{N})-c(\omega_{0},\gamma^{i*}_{\infty}(v^{i}),\mu^{*}_{\infty})\right)=0.
\end{equation*}
According to (Step 2), for $N>\hat{N}(\delta_{\omega,\underline{v}_{N}}\left(\epsilon_{\omega,{\underline{v}_{N}}})\right)$, we have $\mathbb{P}$-a.s.
\begin{flalign*}
\bigg|\frac{1}{N}&\sum_{i=1}^{N}c(\omega_{0},\gamma^{i*}_{N}(v^{i}), \mu^{*}_{N})-c(\omega_{0},\gamma^{i*}_{\infty}(v^{i}), \mu^{*}_{\infty})\bigg|
<\epsilon_{\omega,{\underline{v}_{N}}}.
\end{flalign*}
Following from (c), we can interchange the limit and the integral using the dominated convergence theorem, and the proof is completed.
\end{itemize}
\end{proof}
%\end{theorem}
\begin{remark}\label{rem:3}
One can relax conditions in Theorem \ref{the:6} as follows:
\begin{itemize}
\item [(i)]
relax (a) by considering a sequence of $\epsilon_{N}$-optimal policy, where $\epsilon_{N}$ are non-negative and converges to zero as $N \to \infty$,
\item [(ii)]
relax (c) with a uniform integrability condition which is satisfied if the following expression is finite (see \cite[Theorem 3.5]{Billingsley}),
\begin{flalign*} \label{eq:2.18}
\sup\limits_{N\geq 1}\mathbb{E}\bigg[\bigg\lvert\frac{1}{N}\sum_{i=1}^{N}c\bigg(\omega_{0},\gamma^{i*}_{\infty}(v^{i}),\frac{1}{N}\sum_{i=1}^{{N}}\gamma^{i*}_{\infty}(v^{i})\bigg)\bigg\lvert^{1+\epsilon}\bigg],
\end{flalign*}
for some $\epsilon>0$.
This new condition can be checked in some applications (see Section \ref{sec:5}). The result follows from \cite[Theorem 3.5]{Billingsley},
\item [(iii)]
relax the convergence $\mathbb{P}$-almost surely in (b) by considering convergence in probability, i.e.,
\begin{equation*}
\lim\limits_{N\rightarrow \infty}\mathbb{P}\bigg(\sup\limits_{1\leq i \leq N}\bigg|\gamma^{i*}_{N}(v^{i})-\gamma^{i*}_{\infty}(v^{i})\bigg|\geq\epsilon\bigg)=0,
\end{equation*}
hence similar to the proof of Theorem \ref{the:6}, (Step 1), using continuous mapping theorem (see for example, \cite[page 20]{Billingsley}), we can show that $c(\omega_{0},\gamma^{i*}_{N}(v^{i}),\mu^{*}_{N})$ converges to $c(\omega_{0},\gamma^{i*}_{\infty}(v^{i}),\lim\limits_{N\rightarrow \infty}\mu_{\infty}^{*})$ in probability. Similarly, the result of (Step 2) holds in probability. Using \cite[Theorem 3.5]{Billingsley}, under the uniform integrablity of $X_{N}:=\frac{1}{N}\sum_{i=1}^{N}c\left(\omega_{0},\gamma^{i*}_{\infty}(v^{i}),\frac{1}{N}\sum_{i=1}^{{N}}\gamma^{i*}_{\infty}(v^{i})\right)$ and under the convergence in probability of $X_{N}$ to $X:=\lim\limits_{N\rightarrow \infty}\frac{1}{N}\sum_{i=1}^{N}c\left(\omega_{0},\gamma^{i*}_{\infty}(v^{i}),\frac{1}{N}\sum_{i=1}^{{N}}\gamma^{i*}_{\infty}(v^{i})\right)$, we can conclude that $\mathbb{E}(X_{N}) \to \mathbb{E}(X)$. This relaxation can be useful when the weak law of large numbers can be invoked to check (c), but the strong law of large numbers fails to apply.
\end{itemize}
\end{remark}

We apply the results of this section to two examples in Sections \ref{Ex:1} and \ref{Ex:2}.

In the following section, we show that under symmetry of optimal policies, sufficient conditions of optimality can be satisfied quite effortlessly.

\section{Globally Optimal Policies for Mean-Field Teams}\label{sec:4}
\subsection{Symmetric teams}
In the following, we present sufficient conditions for team optimality in symmetric and mean-field teams. The concept of symmetry has been studied in a variety of contexts; see e.g., \cite{Nash51}, \cite{dasgupta1986existence} and many others.
\begin{definition}(Exchangeable teams)\label{def:sym}\\
An $N$-DM team is \textit{exchangeable} if the value of the expected cost function (see \eqref{eq:1.1}) is invariant under every permutation of policies.
\end{definition}
We note that it is also called \textit{totally symmetric} in a game theoretic context (see for example \cite{dasgupta1986existence}). 
\begin{definition}(Symmetrically optimal teams)\label{Def:ost}\\
A team is \textit{symmetrically optimal}, if for every given policy, there exists an identically symmetric policy (i.e., each DM has the same policy) which performs at least as good as the given policy.
\end{definition}

In the following, we characterize the symmetry of the general setup for ($\mathcal{P}_{N}^{\prime}$) (see \eqref{eq:1.1}) defined in Section \ref{sec:2.1}. Clearly, the result will also hold for the ($\mathcal{P}_{N}$) (see \eqref{eq:2.6.6}) defined in Section \ref{sec:2.2}. First, we recall the definition of an {\it exchangeable} finite set of random variables.
\begin{definition}
Random variables $x^{1},x^{2},\dots,x^{N}$ are \it{exchangeable} if any permutation, $\sigma$, of the set of indexes $\{1,\dots,N\}$ fails to change the joint probability measures of random variables, i.e., $\mathbb{P}(dx^{\sigma(1)},dx^{\sigma(2)},\dots,dx^{\sigma(N)})=\mathbb{P}(dx^{1},dx^{2},\dots,dx^{N})$.
\end{definition}
%Let $J_{N}(\underline{\gamma}_{N}) := E^{\underline{\gamma}_{N}}[c(\omega_{0},\underline{u}_{N})]$, where we define $\omega_{0}$ as the cost function relevant exogenous variable and is contained in $\omega$.
\begin{lemma}\label{lem:4.4.4.4}
For a fixed $N$, consider an $N$-DM team defined as ($\mathcal{P}_{N}^{\prime}$) (see \eqref{eq:1.1}) and let the cost function be  a convex function of $\underline{u}_{N}$ $\mathbb{P}$-almost surely. Assume the cost function is exchangeable $\mathbb{P}$-almost surely with respect to the actions, i.e., for any permutation of indexes, $\sigma$, $\mathbb{P}$-almost surely $c(\omega_{0}, u^{1},\dots, u^{N})=c(\omega_{0}, u^{\sigma(1)},\dots, u^{\sigma(N)})$. If \@$\mathbb{U}$ is convex, and observations of DMs are exchangeable conditioned on $\omega_{0}$, then the team is symmetrically optimal.
\end{lemma}
\begin{proof}
Any permutation of policies does not deviate the value of $J_{N}(\underline{\gamma}_{N})$ since
\begin{flalign}
&J_{N}(\underline{\gamma}^{\sigma}_{N})\nonumber\\&=\int c(\omega_{0}, u^{1},\dots, u^{N})\mathbb{P}_{N}(dv^{1},\dots,dv^{N}|\omega_{0})\nonumber\\
&~~~~~\times 1_{\{(\gamma^{\sigma(1)}(v^{1}),\dots,\gamma^{\sigma(N)}(v^{N}))\}}(du^{1},\dots,du^{N})\mathbb{P}(d\omega_{0})\nonumber\\
&=\int c(\omega_{0}, u^{\sigma(1)},\dots, u^{\sigma(N)})\nonumber\\&~~~~~\times1_{\{(\gamma^{\sigma(1)}(v^{\sigma(1)}),\dots,\gamma^{\sigma(N)}(v^{\sigma(N)}))\}}(du^{\sigma(1)},\dots, du^{\sigma(N)})\nonumber\\
&~~~~~\times\mathbb{P}_{N}(dv^{\sigma(1)},\dots,dv^{\sigma(N)}|\omega_{0})\mathbb{P}(d\omega_{0})\label{eq:4.4.4.4}\\
&=\int c(\omega_{0}, u^{1},\dots, u^{N})1_{\{(\gamma^{1}(v^{1}),\dots,\gamma^{N}(v^{N}))\}}(du^{1},\dots,du^{N})\nonumber\\&~~~~~\times\mathbb{P}_{N}(dv^{1},\dots,dv^{N}|\omega_{0})\mathbb{P}(d\omega_{0})\nonumber\\
&=J_{N}(\underline{\gamma}_{N})\nonumber,
\end{flalign}
 where \eqref{eq:4.4.4.4} follows from the assumption that the cost function is exchangeable with respect to the actions, and the hypothesis that observations of DMs are $\mathbb{P}$-almost surely exchangeable conditioned on the random variable $\omega_{0}$. Let $\underline{\gamma}^{*}_{N}=(\gamma^{1*},\gamma^{2*},\dots, \gamma^{N*})$ be a given team policy for ($\mathcal{P}_{N}^{\prime}$) (see \eqref{eq:1.1}). Consider $\underline{\tilde{\gamma}}_{N}$ as a convex combination of all possible permutations of policies by averaging them, $\sigma \in \Sigma$, where $\Sigma$ is the set of all possible permutation. Since $\mathbb{U}$ is convex, $\underline{\tilde{\gamma}}_{N}$ is a control policy. Following from convexity of the cost function $\mathbb{P}$-almost surely, we have for $\alpha_{\sigma}=\frac{1}{|\Sigma|}$ (where $|\Sigma|$ denotes the cardinality of $\Sigma$),
 \begin{flalign*}
J_{N}(\underline{\tilde{\gamma}}_{N})&:=J_{N}(\sum_{\sigma\in \Sigma}\alpha_{\sigma}\underline{\gamma}^{*,\sigma}_{N}) \leq \sum_{\sigma\in \Sigma}\alpha_{\sigma}J_{N}(\underline{\gamma}^{*,\sigma}_{N})\\&=\sum_{\sigma\in \Sigma}\alpha_{\sigma}J_{N}(\underline{\gamma}^{*}_{N})\ =J_{N}(\underline{\gamma}^{*}_{N}),
 \end{flalign*}
 where the inequality follows from convexity of the cost function $\mathbb{P}$-almost surely for every fixed realization of observations since we have 
 \begin{align*}
& \mathbb{E}\bigg[c\bigg(\omega_{0},\sum_{\sigma\in \Sigma}\alpha_{\sigma}(\underline{\gamma}^{*,\sigma}_{N})^{1}(v^{1}),\dots,\sum_{\sigma\in \Sigma}\alpha_{\sigma}(\underline{\gamma}^{*,\sigma}_{N})^{N}(v^{N})\bigg)\bigg]\\
 &\leq \mathbb{E}\bigg[\sum_{\sigma\in \Sigma}\alpha_{\sigma}c\bigg(\omega_{0},(\underline{\gamma}^{*,\sigma}_{N})^{1}(v^{1}),\dots,(\underline{\gamma}^{*,\sigma}_{N})^{N}(v^{N})\bigg)\bigg]\\
  &=\sum_{\sigma\in \Sigma}\alpha_{\sigma}\mathbb{E}\bigg[c\bigg(\omega_{0},(\underline{\gamma}^{*,\sigma}_{N})^{1}(v^{1}),\dots,(\underline{\gamma}^{*,\sigma}_{N})^{N}(v^{N})\bigg)\bigg],
 \end{align*}
 where $(\underline{\gamma}^{*,\sigma}_{N})^{j}$ denotes the $j$-the component of $\underline{\gamma}^{*,\sigma}_{N}$, and the inequality above follows from Jensen's inequality since the cost function is convex $\mathbb{P}$-almost surely. Hence, the team is symmetrically optimal.
 \end{proof}
 In the following, we present another characterization of symmetrically optimal teams; this looks to be a standard result; however, a proof is included for completeness since we could not find an explicit reference.
\begin{lemma}\label{lem:ex}
For a fixed $N$, consider an $N$-DM team defined as ($\mathcal{P}_{N}^{\prime}$) (see \eqref{eq:1.1}) and let the cost function be  a convex function of $\underline{u}_{N}$ $\mathbb{P}$-almost surely. Assume the set of action space for each DM is convex. If the expected cost function (see \eqref{eq:1.1}) is exchangeable with respect to the policies, then the team is symmetrically optimal.
\end{lemma}
\begin{proof}
Let $\underline{\gamma}^{*}_{N}=(\gamma^{1*},\gamma^{2*},\dots, \gamma^{N*})$ be a given team policy for ($\mathcal{P}_{N}^{\prime}$) (see \eqref{eq:1.1}). According to the definition of exchangeable teams, any permutation of policies, say $\underline{\hat{\gamma}}^{*}_{N}=(\gamma^{i_{1}*},\gamma^{i_{2}*},\dots,\gamma^{i_{N}*})$, fails to change the value of the expected cost function, and hence achieve the same expected cost as the one induced by $\underline{\gamma}^{*}_{N}$. Consider $\underline{\tilde{\gamma}}_{N}$ as a uniform randomization among all possible permutations of optimal policies, since $\mathbb{U}$ is convex then $\underline{\tilde{\gamma}}_{N}$ is a control policy. By convexity of the cost function, through Jensen's inequality, and the fact that any permutation of optimal policies preserves the value of the cost function, we have 
$J_{N}(\underline{\tilde{\gamma}}_{N})\leq J_{N}(\underline{\gamma}^{*}_{N})$. Since $\underline{\tilde{\gamma}}_{N}$ is also identically symmetric, the proof is completed.
\end{proof}
Now, we characterize symmetrically optimal teams for ($\mathcal{P}_{N}$) (see \eqref{eq:2.6.6}). %Let $J(\underline{\gamma})=\limsup\limits_{N\rightarrow \infty} \frac{1}{N} \mathbb{E}^{\underline{\gamma}}(\sum_{i=1}^{N}c(\omega_{0},u^{i},\frac{1}{N}\sum_{p=1}^{{N}}u^{p}))$, where $\omega_{0}$ is the cost function relevant exogenous variable and is contained in $\omega$.
 \begin{theorem}\label{rem:3.3}
 Consider an $N$-DM team defined as ($\mathcal{P}_{N}$) (see \eqref{eq:2.6.6}) in Section \ref{sec:2.2}. Let action spaces be convex and the cost function be convex in the second and third arguments $\mathbb{P}$-almost surely. If observations are exchangeable conditioned on $\omega_{0}$, then the team is symmetrically optimal.
 \end{theorem}
 \begin{proof}
 The cost function defined in ($\mathcal{P}_{N}$) (see \eqref{eq:2.6.6}) is exchangeable in actions, hence under convexity of the action spaces and the cost function and following from the hypothesis that observations are exchangeable condition on $\omega_{0}$, the proof is completed using Lemma \ref{lem:4.4.4.4}.
 \end{proof}
 Theorem \ref{rem:3.3} will be utilized in our analysis to follow. 
\subsection{Optimal solutions for mean-field teams as limits of optimal policies for finite symmetric teams}

In the following, we present results for symmetrically optimal static teams. First, we focus on the case that the observations of decision makers are identical and independent, then we deal with non-identical and dependent observations under additional assumptions.  
As we noted earlier, mean-field games studied in \cite{fischer2017connection} belong to this class in a game theoretic context; in \cite{fischer2017connection} concentration of measures arguments and independence of measurements have been utilized to justify the convergence of equilibria (person-by-person-optimality in the team setup). We also note that \cite{jovanovic1988anonymous} and \cite{mas1984theorem} have considered symmetry conditions for mean-field games. In the context of LQ mean-field teams, \cite{arabneydi2015team} has considered a setup where DMs share the mean-field in the system either completely or partially (through showing that a centralized performance can be attained under the restricted information structure). Also, for the LQ setup under the assumption that DMs apply an identical policy in addition to some technical assumptions, \cite{arabneydi2017certainty} showed that the expected cost achieved by a sub-optimal fully decentralized strategy is on $\epsilon(n)$ neighborhood of the optimal expected cost achieved when mean-field (empirical distribution of states) has been shared, where $n$ is the number of players. In \cite{huang2016linear}, a continuous-time setup with a major agent has been studied. 

\begin{remark}\label{rem:lacker}
We note that, in \cite[Section 2.4]{lacker2017limit}, \cite[Chapter 6 Volume I]{carmona2018probabilistic}, the connection between solutions of $N$-player differential control systems  and solutions of Mckean-Vlasov control problems has been investigated under either the assumption that the information structure is classical (i.e., the problem is centralized) since the controls, $u_{t}^{i}$, for each player are assumed to be progressively measurable with respect to the filtration generated by all initial states, $(X_{0}^{1},\dots, X_{0}^{N})$ and Wiener processes of all DMs ($\{(W_{s}^{1},\dots, W_{s}^{N}), s\leq t\}$), or by imposing structural assumptions on the controllers where controllers assumed to belong to the open-loop class (with their definition being, somewhat non-standard, that $u_{t}^{i}$ are progressively measurable with respect to the filtration generated by initial states and Wiener processes instead of the path of states $X_{s}^{i}$ for $s\leq t$) or to belong to Markovian controllers  (i.e., $u_{t}^{i}=\phi^{i}(t,X_{t}^{i})$ where $\phi^{i}$ are measurable functions) \cite{lacker2017limit},\cite[pages 72-76]{carmona2018probabilistic}. Also, in \cite[Theorem 2.11]{lacker2017limit}, it has been shown that a sequence of relaxed (measure-valued) open-loop $\epsilon_{N}$-optimal policies for $N$-player differential control systems (with only coupling on states) converges to a relaxed open-loop Mckean-Vlasov control optimal solution. Under additional assumptions, the existence of a strong solution and a Markovian optimal solution of McKean-Vlasov solution has been established \cite[Theorem 2.12 and Corollary 2.13]{lacker2017limit}. In the mean-field team setup, under the decentralized information structure, it is not clear apriori whether the limsup of the expected cost function and states of  dynamics for $N$-DM teams converge to the limit. In fact, the information structure of the team problem can break the symmetry and also can prevent establishing a limit theory (for example, by considering a partial sharing of observations between DMs). Here, by focusing on the decentralized setup and by considering mean-field coupling of controls, using a convexity argument and symmetry, we show that a sequence of optimal policies for ($\mathcal{P}_{N}$) converges pointwise to an optimal policy for ($\mathcal{P}_{\infty}$).
\end{remark}
Our next theorem, under the assumption that observations are independent and identically distributed, utilizes a measure concentration argument to establish a convergence result. 
%, we provide the convergence result only under the existence of a convergent sequence of optimal policies (or approximate optimal policies). 

\begin{theorem}\label{the:4.5}
Consider a team defined as  ($\mathcal{P}_{\infty}$) (see \eqref{eq:2.5.5}) with the convex cost function in the second and third arguments $\mathbb{P}$-almost surely. Let the action space be compact and convex for each decision maker, and $v^{i}$s be i.i.d. random variables. If there exists a sequence of optimal policies for ($\mathcal{P}_{N}$) (see \eqref{eq:2.6.6}), $\{\gamma^{*}_{N}\}_N$, which converges (for every decision maker due to the symmetry) pointwise to $\gamma^{*}_{\infty}$ as $N \to \infty$, then $\gamma^{*}_{\infty}$ (which is identically symmetric) is an optimal policy for ($\mathcal{P}_{\infty}$).
\end{theorem}
\begin{proof}
Action spaces and the cost function are convex and following from the hypothesis that $v^{i}$s are i.i.d. random variables (hence they are exchangeable conditioned on $\omega_{0}$) and the result of Theorem \ref{rem:3.3}, one can consider a sequence of $N$-DM teams which are symmetrically optimal that defines ($\mathcal{P}_{N}$) (see \eqref{eq:2.6.6}) and whose limit is identified with $(\mathcal{P}_{\infty})$. Define empirical measures on actions and observation of $Q_{N}(B):=\frac{1}{N}\sum_{i=1}^{N}\delta_{\zeta_{N}^{i}}(B)$, and  $\tilde{Q}_{N}(B):=\frac{1}{N}\sum_{i=1}^{N}\delta_{{\zeta}_{\infty}^{i}}(B)$, where $B \in \mathcal{Z}:=\mathbb{U}\times \mathbb{V}$, $\zeta_{N}^{i}:=(\gamma^{*}_{N}(v^{i}),v^{i})$, ${\zeta}_{\infty}^{i}:=(\gamma^{*}_{\infty}(v^{i}),v^{i})$, and $\delta_{Y}(\cdot)$ is the Dirac measure for any random variable $Y$. In the following, we first show that $\tilde{Q}_{N}$ converges weakly to $Q=Law(\zeta^{i}_{\infty})$ $\mathbb{P}$-almost surely, then we show \eqref{eq:2.14} holds, and we invoke Theorem \ref{the:5}. 
\begin{itemize}[wide = 0pt]
\item [\textbf{(Step 1):}] 
For every $g\in C_{b}(\mathcal{Z})$, where we denote $C_{b}(X)$ as the space of continuous and bounded functions in $X$, we have
\begin{flalign}
&\lim\limits_{N\rightarrow \infty}\mathbb{P}\left(\left|\int g dQ_{N}-\int g d\tilde{Q}_{N}\right|\geq\epsilon\right)\nonumber\\ 
&\ \leq\epsilon^{-1}\lim\limits_{N\rightarrow \infty}\frac{1}{N}\sum_{i=1}^{N}\mathbb{E}\bigg[\bigg|g(\gamma^{*}_{N}(v^{i}),v^{i})-g(\gamma^{*}_{\infty}(v^{i}),v^{i})\bigg|\bigg]\label{eq:4.1.1}\\
&\ =\epsilon^{-1}\lim\limits_{N\rightarrow \infty} \mathbb{E}\bigg[\bigg|g(\gamma^{*}_{N}(v^{i}),v^{i})-g(\gamma^{*}_{\infty}(v^{i}),v^{i})\bigg|\bigg]\label{eq:199.5.1}\\
&\ =\epsilon^{-1}\mathbb{E}\bigg[\lim\limits_{N\rightarrow \infty}\bigg|g(\gamma^{*}_{N}(v^{i}),v^{i})-g(\gamma^{*}_{\infty}(v^{i}),v^{i})\bigg|\bigg]=0 \label{eq:199.5.2},
\end{flalign}
where \eqref{eq:4.1.1} follows from Markov's inequality, the triangle inequality and the definition of the empirical measure, and \eqref{eq:199.5.1} follows from the hypothesis that $v^{i}$s are identical random variables. Since $g$ is bounded and continuous, the dominated convergence theorem implies \eqref{eq:199.5.2}. Hence, for every subsequence there exists a subsubsequence such that $|\int g dQ_{N_{k_{l}}}-\int g d\tilde{Q}_{N_{k_{l}}}|$ converges to zero $\mathbb{P}$-almost surely as $l \to \infty$. On the other hand, since $v^{i}$s are i.i.d. random variables, the strong law of large numbers (SLLN)  implies $\tilde{Q}_{N}$ converges  weakly to $Q$ $\mathbb{P}$-almost surely, that is $|\int g d\tilde{Q}_{N}-\int g d{Q}|$ converges to zero $\mathbb{P}$-almost surely for every $g \in C_{b}(\mathcal{Z})$. Hence, through choosing a suitable subsequence, $Q_{N_{k}}$ converges $\mathbb{P}$-almost sure weakly to $Q$ since for every continuous and bounded function $g$, we have $\mathbb{P}$-a.s.,
\begin{flalign}
&\lim\limits_{N\rightarrow \infty}\left|\int g dQ_{N}-\int g d{Q}\right|\nonumber\\&\leq \lim\limits_{N\rightarrow \infty}\bigg(\bigg|\int g dQ_{N}-\int g d\tilde{Q}_{N}\bigg|+\bigg|\int g d\tilde{Q}_{N}-\int g d{Q}\bigg|\bigg)\nonumber\\&=0\label{eq:20.5}.
\end{flalign}
\item [\textbf{(Step 2):}] 
Following from \cite[Theorem 3.5]{serfozo1982convergence} and \cite[Lemma 1.5]{feinberg2016partially}, or \cite[Theorem 3.1]{langen1981convergence} using the fact that the cost function is non-negative and continuous, we have
\begin{flalign}
&\limsup\limits_{N\rightarrow \infty}\frac{1}{N}\sum_{i=1}^{N}\mathbb{E}\left[c\left(\omega_{0}, \gamma^{*}_{N}(v^{i}),\frac{1}{N}\sum_{i=1}^{N}\gamma^{*}_{N}(v^{i})\right)\right]\nonumber\\
&\geq \liminf\limits_{N\rightarrow \infty}\mathbb{E}\bigg[\mathbb{E}\bigg[\int_{\cal{Z}}c\left(\omega_{0}, u,\int_{\mathbb U}uQ_{N}(du \times \mathbb{V})\right)\nonumber\\
&~~~~~\times Q_{N}(du,dv)\bigg|\omega_{0}\bigg]\bigg]\nonumber\\
&\geq \mathbb{E}\bigg[\mathbb{E}\bigg[\liminf\limits_{N\rightarrow \infty}\int_{\cal{Z}}c\left(\omega_{0}, u,\int_{\mathbb U}uQ_{N}(du \times \mathbb{V})\right)\nonumber\\
&~~~~~\times  Q_{N}(du,dv)\bigg|\omega_{0}\bigg]\bigg]\nonumber\\
&\geq\mathbb{E}\bigg[\mathbb{E}\bigg[\int_{\cal{Z}}c\left(\omega_{0}, u,\int_{\mathbb U}uQ(du \times \mathbb{V})\right)Q(du,dv)\bigg|\omega_{0}\bigg]\bigg] \label{eq:2121},
\end{flalign}
where the first inequality follows from the definition of $Q_{N}$ and replacing limsup by liminf. The second inequality  follows from Fatou's lemma. In the following, we justify \eqref{eq:2121}.  Since $Q_{N}$ converges  weakly to $Q$ $\mathbb{P}$-almost surely, using continuous mapping theorem \cite[page 20]{Billingsley}, we have $Q_{N}(du \times \mathbb{V})$ converges weakly to $Q(du\times \mathbb{V})$ $\mathbb{P}$-almost surely, hence the compactness of $\mathbb{U}$ implies $\int_{\mathbb U}uQ_{N}(du \times \mathbb{V}) \to \int_{\mathbb U}uQ(du \times \mathbb{V})$ $\mathbb{P}$-almost surely, and continuity of the cost function $\mathbb{P}$-almost surely implies $c(\omega_{0}, u,\int_{\mathbb U}uQ_{N}(du \times \mathbb{V}))$ converges to $c\left(\omega_{0}, u,\int_{\mathbb U}uQ(du \times \mathbb{V})\right)$ $\mathbb{P}$-almost surely. Define a non-negative bounded sequence $G_{N}^{M}:=\min\{M,c\left(\omega_{0}, u,\int_{\mathbb U}uQ_{N}(du \times \mathbb{V})\right)\}$, where $G_{M}^{N} \uparrow G^{N}:=c\left(\omega_{0}, u,\int_{\mathbb U}uQ_{N}(du \times \mathbb{V})\right)$ as $M \to \infty$, then we have $\mathbb{P}$-almost surely
\begin{flalign*}
&\liminf\limits_{N\rightarrow \infty}\int_{\cal{Z}}c\left(\omega_{0}, u,\int_{\mathbb U}uQ_{N}(du \times \mathbb{V})\right)Q_{N}(du,dv)\\
&\ =\lim\limits_{M\rightarrow \infty}\liminf\limits_{N\rightarrow \infty}\int_{\cal{Z}}c\left(\omega_{0}, u,\int_{\mathbb U}uQ_{N}(du \times \mathbb{V})\right)Q_{N}(du,dv)\\
&\ \geq \lim\limits_{M\rightarrow \infty}\liminf\limits_{N\rightarrow \infty}\int_{\cal{Z}}G_{N}^{M}Q_{N}(du,dv)\\
&\ = \lim\limits_{M\rightarrow \infty}\int_{\cal{Z}}G^{M}Q(du,dv)\\&\ =\int_{\cal{Z}}c\left(\omega_{0}, u,\int_{\mathbb U}uQ(du \times \mathbb{V})\right)Q(du,dv),
\end{flalign*}
where the first inequality follows from the definition of $G^{M}_{N}$ and the second equality is true using \cite[Theorem 3.5]{serfozo1982convergence} since $G_{N}^{M}$ is bounded (hence it is uniformly $Q_{N}$-integrable) and continuously converges to $G^{M}$, and the monotone convergence theorem implies the last equality. 
Hence, \eqref{eq:2121} holds which implies \eqref{eq:2.14}, and the proof is completed using Theorem \ref{the:5}.
\end{itemize}
\end{proof}

\begin{remark}\label{rem:asym}
The proof above reveals that if $\mathbb{P}$-almost surely the sequence $\{Q_{N}\}_{N}$ converges weakly to $Q$, then Theorem \ref{the:4.5} can be generalized to a class of team problems defined as ($\mathcal{P}_{\infty}$) (see \eqref{eq:2.5.5}) which may include ones with a non-convex cost function and/or the ones with conditionally non-exchangeable observations: This relaxation contains a class of problems (see e.g. Example 4 in Section \ref{Ex:4.4}) where one can consider a sequence of $N$-DM teams which admits asymmetric optimal policies that define ($\mathcal{P}_{N}$) (see \eqref{eq:2.6.6}), but whose limit is identified with $(\mathcal{P}_{\infty})$ under an optimal sequence of policies.
\end{remark}

In the following, we relax the hypothesis that observations of decision makers are independent.
\begin{proposition}\label{prop:4.7}
Consider a team defined as  ($\mathcal{P}_{\infty}$) (see \eqref{eq:2.5.5}) with the convex cost function in the second and third arguments $\mathbb{P}$-almost surely.  Let the action space be compact and convex for each decision maker, and $v^{i}=h(x,z^{i})$, where $z^{i}$s are i.i.d. random variables. If there exists a sequence of optimal policies for ($\mathcal{P}_{N}$) (see \eqref{eq:2.6.6}), $\{\gamma^{*}_{N}\}_N$, which converges pointwise to $\gamma^{*}_{\infty}$ as $N \to \infty$, then $\gamma^{*}_{\infty}$ (which is identically symmetric) is an optimal policy for ($\mathcal{P}_{\infty}$).
\end{proposition}
\begin{proof}
Since $z^{i}$s are i.i.d. random variables, observations, $v^{i}=h(x,z^{i})$, have identical distributions (but are not independent), and similar to the proof of Theorem \ref{the:4.5}, using symmetry, one can show \eqref{eq:199.5.2} holds. In the following, we show \eqref{eq:20.5} and \eqref{eq:2121} hold.
\begin{flalign}
&\scalemath{0.95}{ \lim\limits_{N\rightarrow \infty}\mathbb{P}\bigg(\bigg|\int g d\tilde{Q}_{N}-\int g d{Q}\bigg|\geq\epsilon\bigg)}\nonumber\\
&\scalemath{0.95}{\leq\frac{1}{\epsilon^{2}}\lim\limits_{N\rightarrow \infty}\mathbb{E}\bigg[\bigg|\frac{1}{N}\sum_{i=1}^{N}g(\gamma^{*}_{\infty}(v^{i}),v^{i})-\mathbb{E}(g(\gamma^{*}_{\infty}(v^{1}),v^{1}))\bigg|^{2}\bigg]}\label{eq:5.5.20}\\
&\scalemath{0.95}{ =\lim\limits_{N\rightarrow \infty}(\epsilon)^{-2}\mathbb{E}\bigg[\mathbb{E}\bigg[\bigg|\frac{1}{N}\sum_{i=1}^{N}L(\gamma^{*}_{\infty}(v^{i}),v^{i})\bigg|^{2}\bigg| x\bigg]\bigg]}\label{eq:5.5.21}\\
%&=\lim\limits_{N\rightarrow \infty}(N\epsilon)^{-2}\sum_{i=1}^{N}\mathbb{E}\left(\mathbb{E}\left(\left|L(\gamma^{*}_{\infty}(v^{i}),v^{i})\right|^{2}\mid x\right)\right)\nonumber\\&~~~~~+2\sum_{i=1}^{N-1}\sum_{j=i+1}^{N}\mathbb{E}\left(\mathbb{E}\left(L(\gamma^{*}_{\infty}(v^{i}),v^{i})L(\gamma^{*}_{\infty}(v^{j}),v^{j})|x\right)\right)\label{eq:5.5.22}\\
%&=\lim\limits_{N\rightarrow \infty}(N\epsilon)^{-2}\sum_{i=1}^{N}\mathbb{E}\left(\left|L(\gamma^{*}_{\infty}(v^{i}),v^{i})\right|^{2}\right)
&\scalemath{0.95}{ =0}\label{eq:5.5.23}
\end{flalign}
where $L(\gamma^{*}_{\infty}(v^{i}),v^{i}):=g(\gamma^{*}_{\infty}(v^{i}),v^{i})-\mathbb{E}(g(\gamma^{*}_{\infty}(v^{1}),v^{1})|x)$, and \eqref{eq:5.5.20} follows from Chebyshev's inequality, and \eqref{eq:5.5.21} follows from the law of iterated expectations. The structure $v^{i}=h(x,z^{i})$ implies conditional independence of $v^{i}$s given $x$, hence, using the law of large numbers and since $g \in C_{b}(\mathcal{Z})$, we have \eqref{eq:5.5.23}, and this implies $\tilde{Q}_{N_{k}}$ converges  weakly to $Law(\zeta^{i}_{\infty}|x)$ $\mathbb{P}$-almost surely as $k \to \infty$, hence through choosing a suitable subsequence, $Q_{{N_{k}}_{l}}$ converges $\mathbb{P}$-almost sure weakly to $Q=Law(\zeta^{i}_{\infty}|x)$ as $l \to \infty$ and the rest of the proof to justify \eqref{eq:2121} is the same as that of Theorem \ref{the:4.5}.
\end{proof}
\begin{remark}
Existence of optimal policies for ($\mathcal{P}_{N}$) and dynamic teams satisfying static reduction have been studied in \cite{yuksel2018general} and \cite{gupta2014existence}. In \cite[Theorem 4.8]{yuksel2018general}, the existence of optimal policies achieved under $\sigma$-compactness of each decision maker's action space and under mild conditions on the control law and the cost function. Hence the existence of identically symmetric optimal policies for ($\mathcal{P}_{N}$) (see \eqref{eq:2.6.6}) follows from symmetry and \cite[Theorem 4.8]{yuksel2018general}; thus, the existence result for ($\mathcal{P}_{\infty}$) is obtained under assumptions of Theorem \ref{the:4.5}.
\end{remark} 

In the following, action spaces need not be compact; this is particularly important for LQG models as we will see in the next section.

\begin{theorem} \label{the:4.9.9}
Consider a team defined as  ($\mathcal{P}_{\infty}$) (see \eqref{eq:2.5.5}) with the convex cost function in the second and third arguments $\mathbb{P}$-almost surely.  Let the action spaces be convex for each decision maker. Let $v^{i}$s be i.i.d. random variables. If there exists a sequence of optimal policies for ($\mathcal{P}_{N}$) (see \eqref{eq:2.6.6}), $\{\gamma^{*}_{N}\}_N$, which converges pointwise to $\gamma^{*}_{\infty}$ as $N \to \infty$, and 
\begin{itemize}
\item [(A3)] for some $\delta >0$, $\sup\limits_{N\geq 1}\mathbb{E}(|\gamma^{*}_{N}(v^{1})|^{1+\delta})<\infty$,
\end{itemize}
then $\gamma^{*}_{\infty}$ (which is identically symmetric) is an optimal policy for ($\mathcal{P}_{\infty}$).
\end{theorem}
\begin{proof}
In the following, we just show $\int_{\mathbb U}uQ_{N}(du \times \mathbb{V}) \to \int_{\mathbb U}uQ(du \times \mathbb{V})$ $\mathbb{P}$-almost surely, and the rest of the proof follows from that of Theorem \ref{the:4.5}. We have
\begin{flalign}
\scalemath{0.95}{ \lim\limits_{N\rightarrow \infty}\mathbb{P}}&\scalemath{0.95}{\bigg(\bigg|\int_{\mathbb{U}} u Q_{N}(du\times \mathbb{V})-\int_{\mathbb{U}} u d\tilde{Q}_{N}(du\times \mathbb{V})\bigg|\geq\epsilon\bigg)}\nonumber\\
 &\ \scalemath{0.95}{\leq\epsilon^{-1}\lim\limits_{N\rightarrow \infty}\frac{1}{N}\sum_{i=1}^{N}\mathbb{E}\bigg[\bigg|\gamma^{*}_{N}(v^{i})-\gamma^{*}_{\infty}(v^{i})\bigg|\bigg]}\label{eq:49.9}\\
&\ \scalemath{0.95}{=\epsilon^{-1}\lim\limits_{N\rightarrow \infty} \mathbb{E}\bigg[\bigg|\gamma^{*}_{N}(v^{1})-\gamma^{*}_{\infty}(v^{1})\bigg|\bigg]}\label{eq:49.10}\\
&\ \scalemath{0.95}{=\epsilon^{-1}\mathbb{E}\bigg[\lim\limits_{N\rightarrow \infty}\bigg|\gamma^{*}_{N}(v^{1})-\gamma^{*}_{\infty}(v^{1})\bigg|\bigg]=0}\label{eq:49.11},
\end{flalign}
where \eqref{eq:49.9} follows from Markov's inequality and the triangle inequality, and \eqref{eq:49.10} is true since observations have identical distributions, and \eqref{eq:49.11} follows from the uniform integrability assumption (A3) and the pointwise convergence of $\gamma^{*}_{N}$ using \cite[Theorem 3.5]{Billingsley}. On the other hand, SLLN implies $\mathbb{P}$-almost surely that $\int_{\mathbb U}u{\tilde{Q}}_{N}(du \times \mathbb{V})=\frac{1}{N}\sum_{i=1}^{N}\gamma^{*}_{\infty}(v^{i}) \to \int_{\mathbb U}uQ(du \times \mathbb{V})$, and this completes the proof.
\end{proof}
In the following, we present a result for monotone mean-field coupled teams.
\begin{theorem}\label{the:4.7} Consider a team defined as  ($\mathcal{P}_{\infty}$) (see \eqref{eq:2.5.5}) with the convex cost function in the second and third arguments $\mathbb{P}$-almost surely.  Let the action spaces be convex for each decision maker. Let the cost function be increasing in the last argument, and $v^{i}$s be i.i.d. random variables. If there exists a sequence of optimal policies for ($\mathcal{P}_{N}$), $\{\gamma^{*}_{N}\}_N$ (see \eqref{eq:2.6.6}), which converges pointwise to $\gamma^{*}_{\infty}$ then $\gamma^{*}_{\infty}$ as $N \to \infty$ (which is identically symmetric) is an optimal policy for ($\mathcal{P}_{\infty}$).
\end{theorem}
\begin{proof}
We show \eqref{eq:2.14} holds, then we invoke Theorem \ref{the:5}. We use the same definitions in Theorem \ref{the:4.5} for measures $Q_{N}$ and $Q$. We have
\begin{flalign}
&\scalemath{0.95}{\mathbb{E}\bigg[\mathbb{E}\bigg[\liminf\limits_{N\rightarrow \infty}\int_{\cal{Z}}c\left(\omega_{0}, u,\int_{\mathbb U}uQ_{N}(du \times \mathbb{V})\right)Q_{N}(du,dv)\bigg|\omega_{0}\bigg]\bigg]}\nonumber\\
&\scalemath{0.95}{\geq \mathbb{E}\bigg[\mathbb{E}\bigg[\int_{\cal{Z}}\liminf\limits_{N\rightarrow \infty}c\left(\omega_{0}, u,\int_{\mathbb U}uQ_{N}(du \times \mathbb{V})\right)}\nonumber\\
&~~~~~~~~\scalemath{0.95}{\times Q(du,dv)\bigg|\omega_{0}\bigg]\bigg]}\label{eq:21.5}\\
&\scalemath{0.95}{\geq \mathbb{E}\bigg[\mathbb{E}\bigg[\int_{\cal{Z}}c\left(\omega_{0}, u,\int_{\mathbb U}uQ(du \times \mathbb{V})\right)Q(du,dv)\bigg|\omega_{0}\bigg]\bigg]}, \label{eq:23.5}
\end{flalign}
where \eqref{eq:21.5} follows from  a version of Fatou's lemma in \cite[Theorem 1.1]{feinberg2014fatou}, and \eqref{eq:23.5} is true since from the lower semi-continuity of $\int_{\mathbb U}uQ_{N}(du \times \mathbb{V})$, we have $\liminf\limits_{N\rightarrow \infty}\int_{\mathbb U}uQ_{N}(du \times \mathbb{V})\geq \int_{\mathbb U}uQ(du \times \mathbb{V})$, and continuity and the hypothesis that the cost function is increasing in the last argument imply for all $u \in \mathbb{U}$, $\liminf\limits_{N\rightarrow \infty}c\left(\omega_{0}, u,\int_{\mathbb U}uQ_{N}(du \times \mathbb{V})\right)\geq c\left(\omega_{0}, u,\int_{\mathbb U}uQ(du \times \mathbb{V})\right)$ $\mathbb{P}$-almost surely, and this completes the proof.
% \eqref{eq:21.5} and \eqref{eq:23.5} follows from Skorohod's representation theorem (see for example \cite[page 70]{Billingsley}) which implies the existence of two random variables $X_{N}$ and $X_{\infty}$ with the law of $Q_{N}^{w}$ and $Q^{w}$, $Law(X_{N})=Q^{w}_{N}$ and $Law(X)=Q^{w}$, such that $X_{N} \to X_{\infty}$ $\mathbb{P}$-almost surely and that because $Q^{w}_{N}$ converges weakly to $Q^{w}$. In the above, we defined $h_{N}(X_{N}):=c\left(\omega, X_{N},\int_{\mathbb U}uQ_{N}(du \times \mathbb{V})\right)$ and $h_{\infty}(X_{\infty}):=c\left(\omega, X_{\infty},\int_{\mathbb U}uQ^{\omega}(du \times \mathbb{V})\right)$, and for $\mathbb{P}$-almost all $\omega \in \Omega$
%\begin{equation}
%\liminf\limits_{N\rightarrow \infty}h_{N}(X_{N})=\liminf\limits_{N\rightarrow \infty}c\left(\omega, X_{N},\int_{\mathbb U}uQ_{N}(du \times \mathbb{V})\right)\geq c\left(\omega, X_{\infty},\int_{\mathbb U}uQ^{\omega}(du \times \mathbb{V})\right)=h_{\infty}(X_{\infty}).\label{eq:20.5}
%\end{equation}
%Hence, \eqref{eq:22.5} follows from Fatou's lemma and \eqref{eq:20.5}.
\end{proof}

%\begin{remark}
%One can observe that the presented results in this section also hold for multivariable systems and observations.
%\end{remark}

In the following, observations need not be identical or independent.
\begin{theorem}\label{lem:4.1}
Consider a team defined as  ($\mathcal{P}_{\infty}$) (see \eqref{eq:2.5.5}) with the convex cost function in the second and third arguments $\mathbb{P}$-almost surely.  Let the action spaces be convex for each decision maker. Let (a), and (c) in Theorem \ref{the:6} hold, and let observations be exchangeable conditioned on $\omega_{0}$. Assume there exists a sequence \/ $\{\gamma^{*}_{N}\}_{N}$\@ converges pointwise \/ to\/ $\gamma^{*}_{\infty}$\@ as $N \to \infty$, and 
let $\mathbb{P}$-a.s.
\begin{equation}
\bigg|\gamma^{*}_{N}(v^{i})-\gamma^{*}_{\infty}(v^{i})\bigg|^{2} \leq \frac{f(v^{i})h(N)}{N}\label{eq:22.4.5},
\end{equation}
where $\lim\limits_{N\rightarrow \infty}N^{-1}\sum_{i=1}^{N}f(v^{i})<\infty$ and $\lim\limits_{N \rightarrow \infty}{h(N)}=0$.
Then, a team optimal policy for ($\mathcal{P}_{\infty}$) is symmetrically optimal and an optimal policy is identified as a limit of a sequence of team optimal policies for ($\mathcal{P}_{N}$) (see \eqref{eq:2.6.6}) as $N\to \infty$.
\end{theorem}
\begin{proof}
Following from the result of Theorem \ref{rem:3.3}, one can consider a sequence of $N$-DM teams which are symmetrically optimal that defines ($\mathcal{P}_{N}$) (see \eqref{eq:2.6.6}) and whose limit is identified with $(\mathcal{P}_{\infty})$. Equivalent to (b) in Theorem \ref{the:6}, we can show that $\lim\limits_{N\rightarrow \infty}\sup\limits_{1\leq i\leq N}||\gamma^{*}_{N}(v^{i})-\gamma^{*}_{\infty}(v^{i})||^{2}=0~\mathbb{P}$-almost surely. We have
\begin{flalign*}
\lim\limits_{N\rightarrow \infty}\sup\limits_{1\leq i\leq N}&|\gamma^{*}_{N}(v^{i})-\gamma^{*}_{\infty}(v^{i})|^{2}\\
&\leq \lim\limits_{N\rightarrow \infty}\sum_{i=1}^{N}|\gamma^{*}_{N}(v^{i})-\gamma^{*}_{\infty}(v^{i})|^{2}
\\ &\ \leq\lim\limits_{N\rightarrow \infty}{h(N)}\frac{1}{N}\sum_{i=1}^{N}f(v^{i})=0,
\end{flalign*}
where the last inequality follows from \eqref{eq:22.4.5}. Hence, thanks to Theorem \ref{the:6}, a team optimal policy for ($\mathcal{P}_{\infty}$) is the limit of a sequence of team optimal policies for ($\mathcal{P}_{N}$) (see \eqref{eq:2.6.6}) as $N\to \infty$, and hence a team optimal policy for ($\mathcal{P}_{\infty}$) is symmetrically optimal and the proof is completed.
\end{proof}
\subsection{An existence theorem on globally optimal policies for mean-field team problems}
An implication of our analysis is the following existence result on globally optimal policies for mean-field problems. In Theorem \ref{the:4.5}, we showed that if a pointwise limit as $N \to \infty$ of a sequence of optimal policies for ($\mathcal{P}_{N}$) (see \eqref{eq:2.6.6}) exists, this limit is a globally optimal policy for $(\mathcal{P}_{\infty})$, but under the conditions stated in the following theorem, an existence result also can be established.
In the following, we relax the assumption that there exists a pointwise convergence sequence of optimal policies for ($\mathcal{P}_{N}$) (see \eqref{eq:2.6.6}). For the following theorem, we do not establish the pointwise convergence; but clearly if a sequence of optimal policies for ($\mathcal{P}_{N}$) (see \eqref{eq:2.6.6}) converges pointwise, a global optimal policy exists.  Let $Q_{N}(B):=\frac{1}{N}\sum_{i=1}^{N}\delta_{\zeta_{N}^{i}}(B)$, where $B \in \mathcal{Z}:=\mathbb{U}\times \mathbb{V}$, and $\zeta_{N}^{i}:=(\gamma^{*}_{N}(v^{i}),v^{i})$.
\begin{theorem}\label{the:existence}
Consider ($\mathcal{P}_{\infty}$) (see \eqref{eq:2.5.5}) with the convex cost function in the second and third arguments $\mathbb{P}$-almost surely.  Let the action spaces be convex for each decision maker. Assume further that, without any loss, the optimal control laws can be restricted to those with $\mathbb{E}(\phi_{i}(u^{i}))\leq K$ for some finite $K$, where $\phi_{i}:\mathbb{U}^{i} \to \mathbb{R}_{+}$ is lower semi-continuous. If $v^{i}$s are i.i.d. random variables, then there exists an optimal policy for ($\mathcal{P}_{\infty}$).
\end{theorem}
We note that the limit policy is not necessarily deterministic according to the above result; this interesting discussion is left open for further study.
\begin{proof}
We first show that $\{Q_{N}\}_{N}$ is pre-compact in the product space $(\mathbb{V}\times \mathbb{U})$ equipped with the weak convergence topology for each component. Then, we show that an induced policy by the limit $Q$ achieves lower expected cost than $\limsup\limits_{N\rightarrow \infty}J_{N}(\underline\gamma^{*}_{N})$, and we invoke Theorem \ref{the:5} to complete the proof. Action spaces and the cost function are convex and following from the hypothesis that $v^{i}$s are i.i.d. random variables (hence they are exchangeable conditioned on $\omega_{0}$) and the result of Theorem \ref{rem:3.3}, one can consider a sequence of $N$-DM teams which are symmetrically optimal that defines ($\mathcal{P}_{N}$) (see \eqref{eq:2.6.6}) and whose limit is identified with $(\mathcal{P}_{\infty})$. %Define $\tilde{Q}_{N}(B):=\frac{1}{n}\sum_{i=1}^{n}1_{{\zeta}^{i} \in B}$, where ${\zeta}^{i}:=(\gamma^{*}(v^{i}),v^{i})$ and $\gamma^{*}$ is the limit of a convergent subsequence $\{\gamma^{*}_{n}\}_{n}$.
\begin{itemize}[wide = 0pt]
\item [\textbf{(Step 1):}] 
In the following, we show that for some subsequence $\{Q_{n}\}_{n \in \mathbb{I}}$ converges weakly to $Q$ $\mathbb{P}$-almost surely, that is, $\mathbb{P}$-a.s., for every continuous and bounded function $g$,
\begin{flalign*}
&\lim\limits_{n\rightarrow \infty}\bigg|\int g dQ_{n}-\int g d{Q}\bigg|=0,
\end{flalign*}
where $n \in \mathbb{I}$ is the index set of a converging subsequence. We use the fact that observations are i.i.d. and the space of control policies is weakly compact (see e.g., \cite[proof of Theorem 4.7]{yuksel2018general}). That is because, we can represent the control policy spaces with the space of all joint measures on $(\mathbb{V}^{i} \times \mathbb{U}^{i})$ for each DM with a fixed marginal on $v^{i}$ \cite{YukselSaldiSICON17,BorkarRealization}. Since the team is static, this decouples the policy spaces from the policies of the previous decision makers, and following from the hypothesis on $\phi_{i}$ and the fact that $\nu \to \int \nu(dx)g(x)$ is lower semi-continuous for a continuous function $g$ \cite[proof of Theorem 4.7]{yuksel2018general}, the marginals on $\mathbb{U}^{i}$ will be weakly compact. If the marginals are weakly compact, then the collection of all measures with these weakly compact marginals are also weakly compact (see e.g., \cite[Proof of Theorem 2.4]{yukselSICON2017}) and hence the control policy space is weakly compact. Using Tychonoff's theorem, the countably infinite product space is also compact under the product topology which implies compactness of the space of control policies under the product topology. Hence, there exists a subsequence $\{Q_{n}\}_{n \in \mathbb{I}}$ converges weakly to $Q$ $\mathbb{P}$-almost surely.
\item [\textbf{(Step 2):}] 
Now, we show that \eqref{eq:2.14} holds. We have
%\begin{flalign}
%&\int_{\Omega}\int_{\cal{Z}}c\left(\omega, u,\int_{\mathbb U}uQ^{\omega}(du \times \mathbb{V})\right)Q^{w}(du,dv)\mathbb{P}(d\omega)\nonumber\\
%&= \int_{\Omega}\lim\limits_{n\rightarrow \infty}\int_{\cal{Z}}c\left(\omega, u,\int_{\mathbb U}uQ_{N}(du \times \mathbb{V})\right)Q_{N}(du,dv)\mathbb{P}(d\omega)\label{eq:33.5}\\
%&= \lim\limits_{n\rightarrow \infty}\int_{\Omega}\int_{\cal{Z}}c\left(\omega, u,\int_{\mathbb U}uQ_{N}(du \times \mathbb{V})\right)Q_{N}(du,dv)\mathbb{P}(d\omega)\label{eq:34.5}\\
%&\leq \limsup\limits_{N\rightarrow \infty}\int_{\Omega}\int_{\cal{Z}} c\left(\omega, u,\int_{\mathbb U}uQ_{N}(du \times \mathbb{V})\right)Q_{N}(du,dv)\mathbb{P}(d\omega)\label{eq:35.5}\\
%&= \limsup\limits_{N\rightarrow \infty}\frac{1}{N}\sum_{i=1}^{N}\mathbb{E}\left[c\left(\omega, \gamma^{*}_{N}(v^{i}),\frac{1}{N}\sum_{i=1}^{N}\gamma^{*}_{N}(v^{i})\right)\right]\label{eq:36.5},
%\end{flalign}
\begin{flalign}
&\scalemath{0.95}{\mathbb{E}\bigg[\mathbb{E}\bigg[\int_{\cal{Z}}c\left(\omega_{0}, u,\int_{\mathbb U}uQ(du \times \mathbb{V})\right)Q(du,dv)\bigg|\omega_{0}\bigg]\bigg]}\nonumber\\
&\scalemath{0.95}{=  \lim\limits_{M\rightarrow \infty}\mathbb{E}\bigg[\mathbb{E}\bigg[\int_{\cal{Z}}\min\bigg\{M,c\left(\omega_{0}, u,\int_{\mathbb U}uQ(du \times \mathbb{V})\right)\bigg\}}\nonumber\\&~~~~~\scalemath{0.95}{\times Q(du,dv)\bigg|\omega_{0}\bigg]\bigg]}\label{eq:37.5.1}\\
&\scalemath{0.95}{=  \lim\limits_{M\rightarrow \infty}\mathbb{E}\bigg[\mathbb{E}\bigg[\lim\limits_{n\rightarrow \infty}\int_{\cal{Z}}\min\bigg\{M,c\left(\omega_{0}, u,\int_{\mathbb U}uQ_{n}(du \times \mathbb{V})\right)\bigg\}}\nonumber\\&~~~~~\scalemath{0.95}{\times Q_{n}(du,dv)\bigg|\omega_{0}\bigg]\bigg]}\label{eq:33.5.1}\\
&\scalemath{0.95}{=  \lim\limits_{M\rightarrow \infty}\lim\limits_{n\rightarrow \infty}\mathbb{E}\bigg[\mathbb{E}\bigg[\int_{\cal{Z}}\min\bigg\{M,c\left(\omega_{0}, u,\int_{\mathbb U}uQ_{n}(du \times \mathbb{V})\right)\bigg\}}\nonumber\\&~~~~~\scalemath{0.95}{\times Q_{n}(du,dv)\bigg|\omega_{0}\bigg]\bigg]}\label{eq:34.5.1}\\
&\scalemath{0.92}{ \leq  \lim\limits_{M\rightarrow \infty}\limsup\limits_{N\rightarrow \infty}\mathbb{E}\bigg[\mathbb{E}\bigg[\int_{\cal{Z}} \min\bigg\{M,c\left(\omega_{0}, u,\int_{\mathbb U}uQ_{N}(du \times \mathbb{V})\right)\bigg\}}\nonumber\\&~~~~~\scalemath{0.95}{\times Q_{N}(du,dv)\bigg|\omega_{0}\bigg]\bigg]}\label{eq:35.5.1}\\
&\scalemath{0.95}{\leq \limsup\limits_{N\rightarrow \infty}\frac{1}{N}\sum_{i=1}^{N}\mathbb{E}\left[c\left(\omega_{0}, \gamma^{*}_{N}(v^{i}),\frac{1}{N}\sum_{i=1}^{N}\gamma^{*}_{N}(v^{i})\right)\right]}\label{eq:36.5.1},
\end{flalign}
where \eqref{eq:37.5.1} follows from the monotone convergence theorem. Since $\{Q_{n}\}_{n \in I}$ converges weakly to $Q$ $\mathbb{P}$-almost surely, we have by continuous mapping theorem (by considering a projection to the first component) $\int_{\mathbb U}uQ_{n}(du \times \mathbb{V}) \to \int_{\mathbb U}uQ(du \times \mathbb{V})$ $\mathbb{P}$-almost surely. Following from (Step 1), \eqref{eq:33.5.1} follows from \cite[Theorem 3.5]{serfozo1982convergence}. That is because, the cost function is continuous in actions, and $\min\{M,c\left(\omega_{0}, u,\int_{\mathbb U}uQ_{n}(du \times \mathbb{V})\right)\}$ is continuously converges in $u$, $\min\{M,c(\omega_{0}, u_{n},\int_{\mathbb U}uQ_{n}(du \times \mathbb{V}))\} \to \min\{M,c(\omega_{0}, u,\int_{\mathbb U}uQ(du \times \mathbb{V}))\}$ where $u_{n} \to u$ as $n \to \infty$. Equality \eqref{eq:34.5.1} follows from the dominated convergence theorem since $\min\{M,c\left(\omega_{0}, u,\int_{\mathbb U}uQ_{N}(du \times \mathbb{V})\right)\}$ is bounded, and \eqref{eq:35.5.1} is true since limsup is the greatest convergent subsequence limit for a bounded sequence. Finally, \eqref{eq:36.5.1} follows from the definition of empirical measures and since for every $M$, $\min\{M,c\left(\omega_{0}, u,\int_{\mathbb U}uQ_{N}(du \times \mathbb{V})\right)\} \leq c\left(\omega_{0}, u,\int_{\mathbb U}uQ_{N}(du \times \mathbb{V})\right)$; hence, following from Theorem \ref{the:5}, the randomized limit policy through subsequence is a globally optimal for ($\mathcal{P}_{\infty}$). 
\end{itemize}
\end{proof}

We apply the results of this section in Section \ref{Ex:3}.

%We will use \cref{the:4.5} and \cref{lem:4.1} in Section \ref{Ex:3} below.
\section{Examples}\label{sec:5}
In the following, we present a number of examples to demonstrate results in previous sections. First, we consider LQG and LQ static teams with coupling between states, then we consider LQG symmetric static teams with coupling between control actions. Moreover, we investigate dynamic infinite-horizon average cost LQG teams with the classical information structure.

%\sy{In the following, Example 1 is an application of Theorem X, Example 2 is one of Theorem X, Example 3 of Theorem X and finally Example 4 of Theorem X.}

\subsection{ Example 1, Static quadratic Gaussian teams with coupling between states}\label{Ex:1}
Consider the following observation scheme,
\begin{equation}\label{eq:2.31}
v^{i}=x^{i}+z^{i},
\end{equation}
where $\{z^{i}\}_{i \in \mathbb{N}}$ and $\{x^{i}\}_{i \in \mathbb{N}}$ are i.i.d. zero mean Gaussian random variables. Let $\{z^{i}\}_{i \in \mathbb{N}}$ be independent of $\{x^{i}\}_{i \in \mathbb{N}}$. The expected cost function is defined as
\begin{equation}\label{eq:2.32}
J(\underline\gamma)=\limsup\limits_{N\rightarrow \infty} \frac{1}{N} \mathbb{E}^{\underline{\gamma}}\bigg[\sum_{i=1}^{N}R(u^{i})^{2}+Q(u^{i}-x^{i}-\mu_{N})^{2}\bigg],
\end{equation}
where $\mu_{N}:=\frac{1}{N}\sum_{k=1}^{N}x^{k}$. Let $R$ be a positive number and $Q$ be a non-negative number.
\begin{theorem}\label{the:7}
For LQG static teams as formulated above, under the measurement scheme \eqref{eq:2.31},\/ $\gamma^{i*}_{\infty}(v^{i})$\@ is globally optimal for ($\mathcal{P}_{\infty}$) achieved as the limit $N \to \infty$ of\/ $\gamma^{i*}_{N}(v^{i})$\@, an optimal solution for ($\mathcal{P}_{N}$).
\end{theorem}
\begin{proof}
We invoke Theorem \ref{the:6} to prove the theorem. The stationary policy (see Definition \ref{def:sta}) is obtained as
\begin{flalign*}
\gamma^{i*}_{N}&=(R+Q)^{-1}Q(1+\frac{1}{N})\mathbb{E}(x^{i}|v^{i}),
%=(R+Q)^{-1}Q\left(\mathbb{E}(x^{i}|v^{i})+\frac{1}{N}\sum_{p=1}^{N}\mathbb{E}(x^{p}|v^{i})\right)
\end{flalign*}
where the equality follows from the assumption that $x^{i}$s are independent of $z^{i}$s and $x^{k}$s, $k\not=i$ for every $i=1,2,...,N$ and the assumption that random variables are mean zero. Following from \cite{KraMar82}, stationary policies are team optimal for ($\mathcal{P}_{N}$) in this formulation. We have $\gamma^{i*}_{\infty}(v^{i})=(R+Q)^{-1}Q\mathbb{E}(x^{i}|v^{i})$.
Since $v^{i}$s are zero mean Gaussian random variables, we have $\mathbb{E}(x^{i}|v^{i})=\Sigma_{x^{i}v^{i}}\Sigma_{v^{i}v^{i}}^{-1}v^{i}:=Kv^{i}$, where $\Sigma_{XY}$ is defined as a covariance of two random variables $X$ and $Y$.%Now, we require to check that (3.64) is a stationary solution for (3.62).
\ We have $\mathbb{P}$-a.s.,
\begin{flalign}\label{eq:2.42}
\sup\limits_{1 \leq i\leq N}|\gamma^{i*}_{N}(v^{i})-\gamma^{i*}_{\infty}(v^{i})|
&=\frac{Q}{R+Q}\sup\limits_{1 \leq i\leq N}|\frac{1}{N}\mathbb{E}(x^{i}|v^{i})|\nonumber\\&=\frac{KQ}{R+Q}\sup\limits_{1 \leq i\leq N}|\frac{1}{N}v^{i}| \xrightarrow[N \rightarrow \infty]{}0,
\end{flalign}
where \eqref{eq:2.42} follows from
\begin{flalign*}
&\lim\limits_{N\rightarrow \infty}\sup\limits_{1\leq i \leq N}\frac{1}{N^{2}}(v^{i})^{2}\leq \lim\limits_{N\rightarrow \infty}\frac{1}{N^{2}}\sum_{i=1}^{N}(v^{i})^{2}=0 ~\mathbb{P}-a.s,
\end{flalign*}
where the first inequality is true since $(v^{i})^{2}$s are non-negative, and equality follows from the strong law of large numbers (SLLN) since $v^{i}$s are i.i.d. and have a finite variance, hence, (b) holds. One can show that the condition in Remark \ref{rem:3}(ii) holds since $v^{i}$s and $x^{i}$s are i.i.d. random variables, hence Theorem \ref{the:6} completes the proof.
\end{proof}
\subsection{ Example 2, Static non-Gaussian teams with coupling between states}\label{Ex:2}
Let the observation scheme be \eqref{eq:2.31}, where $\{z^{i}\}_{i \in \mathbb{N}}$ and $\{x^{i}\}_{i \in \mathbb{N}}$ are i.i.d. zero mean random variables with finite variance. Let $\{z^{i}\}_{i \in \mathbb{N}}$ be independent of $\{x^{i}\}_{i \in \mathbb{N}}$.
The expected cost function is defined as \eqref{eq:2.32}.
Let $R$ be a positive number and $Q$ be a non-negative number.
\begin{theorem}\label{the:8}
For LQ static teams as formulated above, under the measurement scheme \eqref{eq:2.31},\/ $\gamma^{k*}_{\infty}(v^{k})=(R+Q)^{-1}Q\mathbb{E}(x^{k}|v^{k})$\@ is globally optimal for ($\mathcal{P}_{\infty}$) and is obtained as the limit of $\gamma^{k*}_{N}(v^{k})$ as $N \to \infty$.
\end{theorem}
\begin{proof}
In the following, we use both Theorem \ref{the:4} and Theorem \ref{the:6}. Clearly, (A1) holds, we show that (A2) holds,
\begin{flalign}
&\limsup\limits_{N\rightarrow \infty}\frac{1}{N}\mathbb{E}\bigg[\sum_{i=1}^{N}(\gamma^{i*}_{\infty}(v^{i}))^{2}R+Q(\gamma^{i*}_{\infty}(v^{i})-x^{i}-\mu_{N})^{2}\bigg]\nonumber\\
&=\limsup\limits_{N\rightarrow \infty}\frac{1}{N}\mathbb{E}\bigg[\sum_{i=1}^{N}\frac{-Q^{2}}{Q+R}\mathbb{E}^{2}(x^{i}|v^{i})(1+\frac{2}{N})(x^{i}+\mu_{N})^{2}Q\bigg]\label{eq:2.48}\\
%&\ =\limsup\limits_{N\rightarrow \infty}\frac{1}{N}\mathbb{E}\left[\sum_{i=1}^{N}\frac{-Q^{2}}{Q+R}\mathbb{E}^{2}(x^{i}|v^{i})(1+\frac{2}{N})+Q\sum_{i=1}^{N}(1+\frac{2}{N})\sigma^2+\sum_{i=1}^{N}\frac{Q}{N^{2}}\sum_{k=1}^{N}\sigma^{2}\right]\label{eq:2.49}\\
&\leq \limsup\limits_{N\rightarrow \infty}\frac{1}{N}\mathbb{E}\bigg[\sum_{i=1}^{N}\frac{-Q^{2}}{Q+R}\mathbb{E}^{2}(x^{i}|v^{i})\bigg]+\lim\limits_{N\rightarrow \infty}\frac{Q(N+3)\sigma^2}{N}\label{eq:2.51}\\
%&\ =\frac{-Q^{2}}{Q+R}\liminf\limits_{N\rightarrow \infty}\mathbb{E}\left[\frac{1}{N}\left(\sum_{i=1}^{N}(Y^{i})\right)\right]+Q\sigma^{2}\label{eq:2.54}\\
&=\frac{-Q^{2}}{Q+R}\mathbb{E}\bigg[\mathbb{E}^{2}(x^{1}|v^{1})\bigg]+Q\sigma^{2}\label{eq:2.55},
\end{flalign}
where \eqref{eq:2.48} follows from $\mathbb{E}\left(\mathbb{E}(x^{i}|v^{i})(x^{i}+\mu_{N})\right)=\mathbb{E}\left(\mathbb{E}\left(\mathbb{E}(x^{i}|v^{i})(x^{i}+\mu_{N})|v^{i}\right)\right)=(1+\frac{1}{N})\mathbb{E}\left(\mathbb{E}^{2}(x^{i}|v^{i})\right)$, and \eqref{eq:2.51} is true since $x^{i}$ and $z^{i}$ are i.i.d. random variables and $\limsup\limits_{N\rightarrow \infty}a_{N}+\limsup\limits_{N\rightarrow \infty}b_{N} \geq \limsup\limits_{N\rightarrow \infty} (a_{N}+b_{N})$. We can justify \eqref{eq:2.55} by defining $Y^{i}:=(\mathbb{E}(x^{i}|v^{i}))^{2}$, and since $Y^{i}$s are measurable functions of $\{v^{i}\}_{i \geq 1}$, and $v^{i}$s and $x^{i}$s are i.i.d., $Y^{i}$s are i.i.d. random variables. %We have $\{Y^{i}\}_{i \geq 1}$ satisfy SLLN since $\mathbb{E}(Y^{1})=\mathbb{E}\left((\mathbb{E}(x^{1}|v^{1}))^{2}\right)\leq \mathbb{E}(\mathbb{E}\left((x^{1})^{2}|v^{1})\right)=\mathbb{E}\left((x^{1})^{2}\right) <\infty$ (see \cite{BillingsleyProbMeasure}, Theorem 6.1), where the first inequality follows from conditional Jensen's inequality (see \cite{Dud02}, page 349). 
Similarly, one can show the other side direction for liminf.
%\begin{flalign*}
%\liminf\limits_{N\rightarrow \infty}\frac{1}{N}\mathbb{E}&(\sum_{i=1}^{N}(\gamma^{i*}_{\infty}(v^{i}))^{2}(R+Q)\\&-2\gamma^{i*}_{\infty}(v^{i})(x^{i}+\mu_{N})Q+(x^{i}+\mu_{N})^{2}Q)\\&\geq \frac{-Q^{2}}{Q+R}\mathbb{E}\left(\mathbb{E}^{2}(x^{1}|v^{1})\right)+Q\sigma^{2}
%\end{flalign*}
%\begin{flalign*}
%&\liminf\limits_{N\rightarrow \infty}\frac{1}{N}\mathbb{E}\left(\sum_{i=1}^{N}(\gamma^{i*}_{\infty}(v^{i}))^{2}(R+Q)-2\gamma^{i*}_{\infty}(v^{i})(x^{i}+\mu_{N})Q+(x^{i}+\mu_{N})^{2}Q\right)\\
%&\ =\liminf\limits_{N\rightarrow \infty}\frac{1}{N}\left[\mathbb{E}(\sum_{i=1}^{N}\frac{-Q^{2}}{Q+R}\mathbb{E}^{2}(x^{i}|v^{i})(1+\frac{2}{N})+(x^{i}+\mu_{N})^{2}Q\right]\\
%&\ \geq \liminf\limits_{N\rightarrow \infty}\frac{1}{N}\mathbb{E}\left[\sum_{i=1}^{N}\frac{-Q^{2}}{Q+R}\mathbb{E}^{2}(x^{i}|v^{i})\right]+\liminf\limits_{N\rightarrow \infty}\frac{1}{N}Q(N+3)\sigma^2\\
%&\ =\frac{-Q^{2}}{Q+R}\limsup\limits_{N\rightarrow \infty}\frac{1}{N}\mathbb{E}\left(\sum_{i=1}^{N}Y^{i}\right)+Q\sigma^{2}\\
%&\ =\frac{-Q^{2}}{Q+R}\mathbb{E}\left(\mathbb{E}^{2}(x^{i}|v^{i})\right)+Q\sigma^{2}.
%\end{flalign*}
Hence (A2) is satisfied. Now, we check \eqref{eq:2.2}, for every $\gamma^{k}_{\infty}$ with $J(\underline{\gamma}_{\infty})<\infty$,
\begin{flalign}
&\limsup\limits_{N\rightarrow \infty} \frac{1}{N}\mathbb{E}\bigg(\sum_{i=1}^{N}\sum_{k=1}^{N}c_{u^{k}}(\omega_{0},\gamma^{i*},\mu^{*})(m_{k})\bigg)\nonumber\\
%&=\limsup\limits_{N\rightarrow \infty}\frac{1}{N}\mathbb{E}(\sum_{k=1}^{N}[2\gamma^{k*}_{\infty}(v^{k})(R+Q)\nonumber\\&~~~~~~-2Q(x^{k}+\frac{1}{N}\sum_{p=1}^{N}x^{p})](m_{k}))\nonumber\\
%&=\limsup\limits_{N\rightarrow \infty}\frac{2Q}{N}\mathbb{E}\left(\sum_{k=1}^{N}\left[\mathbb{E}(x^{k}|v^{k})-(x^{k}+\mu_{N})\right]\left(m_{k}\right)\right)\nonumber\\
&=\limsup\limits_{N\rightarrow \infty}\frac{2Q}{N}\sum_{k=1}^{N}\mathbb{E}\left(\mathbb{E}(x^{k}(m_{k})|v^{k})\right)-\mathbb{E}\left((x^{k}+\mu_{N})(m_{k})\right)\label{eq:2.58}\\
&=\limsup\limits_{N\rightarrow \infty}\frac{-2Q}{N}\sum_{k=1}^{N}\mathbb{E}\left(\mu_{N}\left(m_{k}\right)\right)\label{eq:2.4.5.8}\\
%&\ =\limsup\limits_{N\rightarrow \infty}\frac{-2Q}{N^{2}}\sum_{k=1}^{N}\mathbb{E}\left[\sum_{p=1,p\not=k}^{N}x^{p}\left(m_{k}\right)+x^{k}\left(m_{k}\right)\right]\nonumber\\
&=2Q\limsup\limits_{N\rightarrow \infty}\frac{1}{N^{2}}\sum_{k=1}^{N}\mathbb{E}\left[x^{k}\gamma^{k*}_{\infty}(v^{k})\right]-\mathbb{E}\left[x^{k}\gamma^{k}_{\infty}(v^{k})\right]\label{eq:2.4.5.9}\\
%&\ =2\frac{Q^{2}}{R+Q}\lim\limits_{N\rightarrow \infty}\frac{1}{N^{2}}N\mathbb{E}\left(x^{1}\mathbb{E}(x^{1}|v^{1})\right)\nonumber\\&~~~~~~+2Q\limsup\limits_{N\rightarrow \infty}-\frac{1}{N^{2}}\sum_{k=1}^{N}\mathbb{E}\left(x^{k}\gamma^{k}_{\infty}(v^{k})\right)\nonumber\\
&\ =-2Q\liminf\limits_{N\rightarrow \infty}\frac{1}{N^{2}}\sum_{k=1}^{N}\mathbb{E}\left(x^{k}\gamma^{k}_{\infty}(v^{k})\right)\label{eq:2.63}\\
&\geq-2Q\sigma\liminf\limits_{N\rightarrow \infty}\frac{1}{N^{2}}\sum_{k=1}^{N}\sqrt{\mathbb{E}\left[(\gamma^{k}_{\infty}(v^{k}))^{2}\right]}\label{eq:2.64}\\
&\geq-2Q\sigma\liminf\limits_{N\rightarrow \infty}\sup\limits_{1 \leq k \leq N}\sqrt{\frac{\mathbb{E}\left[(\gamma^{k}_{\infty}(v^{k}))^{2}\right]}{N^{2}}}=0\label{eq:2.66},
\end{flalign}
where measurability of $m_{k}:=\gamma^{k}_{\infty}(v^{k})-\gamma^{k*}_{\infty}(v^{k})$ with respect to the $\sigma$-field generated by $v^{k}$ implies \eqref{eq:2.58}, and \eqref{eq:2.4.5.8} follows from the iterated expectations property. Since $x^{p}$s are mean zero and independent of $v^{k}$ for $k\not=p$, we have \eqref{eq:2.4.5.9}, and \eqref{eq:2.63} follows from the fact that $\gamma^{k*}_{\infty}$ is independent of $k$, and since $v^{k}$ and $x^{k}$ are i.i.d. random variables. Moreover, $J(\underline{\gamma}_{\infty})<\infty$, so $\mathbb{E}(\gamma^{k}_{\infty}(v^{k}))\leq \mathbb{E}\left((\gamma^{k}_{\infty}(v^{k}))^{2}\right)<\infty$, and Cauchy\textendash Schwarz inequality implies \eqref{eq:2.64}, and \eqref{eq:2.66} follows from
\begin{flalign}
&\liminf\limits_{N\rightarrow \infty}\sup\limits_{1 \leq k \leq N}{\frac{\mathbb{E}\left[(\gamma^{k}_{\infty}(v^{k}))^{2}\right]}{N^{2}}}\nonumber\\
&\leq \liminf\limits_{N\rightarrow \infty}\frac{1}{N^{2}}\sum_{k=1}^{N}{\mathbb{E}\left[(\gamma^{k}_{\infty}(v^{k}))^{2}\right]}=0 \label{eq:2.68},
\end{flalign}
where \eqref{eq:2.68} is true since $\mathbb{E}\left[(\gamma^{k}_{\infty}(v^{k}))^{2}\right]\geq 0$ and $\limsup\limits_{N\rightarrow \infty}\frac{1}{N}\mathbb{E}\left(\sum_{k=1}^{N}\left(\gamma^{k}_{\infty}(v^{k})\right)^{2}R\right) \leq J(\underline{\gamma}_{\infty})<\infty$. Thus, \eqref{eq:2.2} is satisfied and Theorem \ref{the:4} completes the proof.
\\One can also invoke Theorem \ref{the:6} to complete the proof.  One can show that the condition in Remark \ref{rem:3}(ii) holds since $v^{i}$s and $x^{i}$s are i.i.d. random variables. We only justify (b). Stationary policy is team optimal for ($\mathcal{P}_{N}$) in this formulation \cite{KraMar82}, hence $\gamma^{i*}_{N}(v^{i})=(R+Q)^{-1}Q(1+\frac{1}{N})\mathbb{E}(x^{i}|v^{i})$, so we need to show that
\begin{flalign*}
\lim\limits_{N\rightarrow \infty}&\sup\limits_{1\leq i \leq N}\left|\gamma^{i*}_{N}(v^{i})-\gamma^{i*}_{\infty}(v^{i})\right|=0~\mathbb{P}-a.s,
%\\& =\lim\limits_{N\rightarrow \infty}\sup\limits_{1\leq i \leq N}\left|\frac{1}{N}\mathbb{E}(x^{i}|v^{i})\right|=0~\mathbb{P}-a.s,
\end{flalign*}
Equivalently, we can show that $\mathbb{P}$-a.s
\begin{flalign*}
&\scalemath{0.95}{\lim\limits_{N\rightarrow \infty}\sup\limits_{1\leq i \leq N}\frac{1}{N^{2}}\left(\mathbb{E}(x^{i}|v^{i})\right)^{2}\leq \lim\limits_{N\rightarrow \infty}\frac{1}{N^{2}}\sum_{i=1}^{N}\left(\mathbb{E}(x^{i}|v^{i})\right)^{2}=0}, \label{eq:2.69}
\end{flalign*}
where the first inequality is true since $\left(\mathbb{E}(x^{i}|v^{i})\right)^{2}$s are non-negative, and equality follows from SLLN since
\begin{flalign*}
\mathbb{E}\left((\mathbb{E}(x^{i}|v^{i}))^{2}\right)& = \mathbb{E}\left((x^{i})^{2}\right)-\mathbb{E}\left((x^{i}-\mathbb{E}(x^{i}|v^{i}))^{2}\right)<\infty,
%&\leq \mathbb{E}((x^{i})^{2}) <\infty,
\end{flalign*}
and $(\mathbb{E}\left(x^{i}|v^{i})\right)^{2}$ are i.i.d. sequence of random variables since $v^{i}$s are i.i.d. random variables and the proof is completed.
\end{proof}
\subsection{Example 3, LQG symmetric teams with coupling between control actions}\label{Ex:3}
Let
\begin{equation}
\underline{v}^{i}=H^{i}\underline{x}+\underline{z}^{i},\label{eq:1.8}
\end{equation}
where $\{\underline{z}^{i}\}_{{i} \in \mathbb{N}}$ is independent zero mean Gaussian random vectors also independent of $\underline{x}$, with covariance $\Sigma_{jj}=N^{0}>0$. Define $\omega=(\underline{x},\underline{z}^{1},\underline{z}^{2},\dots)$, and $\omega_{0}:=\underline{x}$ where $\underline{x}$ is a Gaussian random vector with covariance $\mathbb{E}(\underline{x}\underline{x}^{T})=\Sigma_{00}$.
Let
\begin{flalign}
&J(\underline\gamma)=\limsup\limits_{N\rightarrow \infty} \frac{1}{N} \mathbb{E}^{\underline{\gamma}}\bigg[\sum_{i=1}^{N}(u^{i})^{T}Ru^{i}-2\sum_{i=1}^{N}(u^{i})^{T}D\nonumber\\&\times(\underline{x}+\frac{1}{N}\sum_{k=1}^{N}u^{k})+(\underline{x}+\frac{1}{N}\sum_{k=1}^{N}u^{k})^{T}Q(\underline{x}+\frac{1}{N}\sum_{k=1}^{N}u^{k})\bigg],\label{eq:1.9}
\end{flalign}
where $R$ is an appropriate dimension positive definite matrix and $D$, and $Q$ are appropriate dimension positive semi-definite matrices, and $R>2D$.
In the following, we follow steps in \cite[Theorem 2.6.8]{YukselBasarBook} to obtain optimal policies for ($\mathcal{P}_{N}$).
%Note that we can also represent \eqref{eq:1.9} as a matrix form as follows:
%\begin{equation}
%J(\underline\gamma)=\limsup\limits_{N\rightarrow \infty} \frac{1}{N} \mathbb{E}^{\underline{\gamma}}[\underline{u}_{N}^{T}F\underline{u}_{N}-2\underline{u}_{N}^{T}\hat{D}(x+\mu_{N})+(x+\mu_{N})^{T}Q(x+\mu_{N})]\label{eq:1.10}
%\end{equation}
%where $\mu_{N}=\displaystyle\frac{1}{N}\sum_{k=1}^{N}u^{k}$, and $F=R\otimes I_{N \times N}$, $\hat{D}=D\otimes[1 1 ... 1]^{T}$. $I_{N \times N}$ is $N \times N$ identical matrix, and $[1 1 ... 1]^{T}$ is $N \times 1$ vector which all arrays are one.
\begin{lemma}
 Consider an $N$-DM LQG team formulated above,
%\begin{flalign}
%J_{N}(\underline\gamma_{N})=\frac{1}{N} &\mathbb{E}^{\underline{\gamma}}[\sum_{i=1}^{N}(u^{i})^{T}Ru^{i}-2\sum_{i=1}%^{N}(u^{i})^{T}D(\underline{x}+\frac{1}{N}\sum_{k=1}^{N}u^{k})\nonumber\\ &+(\underline{x}+\frac{1}{N}\sum_{k=1}^{N}%u^{k})^{T}Q(\underline{x}+\frac{1}{N}\sum_{k=1}^{N}u^{k})]\label{eq:1.11},
%\end{flalign}
under the measurement scheme \eqref{eq:1.8}, the global optimal policy  for ($\mathcal{P}_{N}$) is linear, i.e., $\gamma^{k*}_{N}(v^{k})=\pi^{k}_{N}v^{k}$. Here, $\pi^{k}_{N} \in \mathcal{M}_{n,m}(\mathbb{R})$, $n \times m$ real-valued matrix, is obtained by solving the following parallel update scheme,
\begin{flalign}
&\pi^{k}_{N,(i)}=-L_{N}\bigg[S^{k}+\frac{1}{N}\sum_{p=1,p\not=k}^{N}\pi^{p}_{N,(i)}{H^{p}S^{k}}\bigg],\label{eq:1.12}
\end{flalign}
where $L_{N}:=(R+\frac{Q}{N^{2}}-\frac{2D}{N})^{-1}(\frac{Q}{N}-D)$, $S^{k}:={\Sigma_{00}(H^{k})^{T}}({H^{k}\Sigma_{00}(H^{k})^{T}+\Sigma_{kk}})^{-1}$ and the initial points of the iterations are considered as zero functions.
\end{lemma}
\begin{proof}
By Definition \ref{def:sta}, stationary policies satisfy the following equality for $k=1,\dots,N$,
\begin{flalign}
&M\gamma^{k*}_{N}(v^{k})+\bigg(\frac{Q}{N}-D\bigg)\nonumber\\&\times\bigg[\mathbb{E}(\underline{x}|v^{k})+\frac{1}{N}\sum_{p=1,p\not=k}^{N}\mathbb{E}\bigg(\gamma^{p*}_{N}(v^{p})|v^{k}\bigg)\bigg]=0,\label{eq:1.13}
\end{flalign}
where $M:=R+\frac{Q}{N^{2}}-\frac{2D}{N}$, and \eqref{eq:1.13} can be rewritten as $P\underline{\hat{R}}\underline{\gamma}^{*}_{N}(\underline{v})+Pr(\omega)=0$, where $P$ is a block diagonal matrix with $ii$-th block $P_{ii}\beta^{i}(\omega):=\mathbb{E}(\beta^{i}(\omega)|v^{i})$, $\underline{\hat{R}}$ is a matrix where $\hat{R}_{ii}:=M$ and $\displaystyle\hat{R}_{ij}:=\frac{1}{N}(\frac{Q}{N}-D)$ for every $i,j=1,...,N$, $j\not=i$, and $r(\omega)=\underline{x}$.  Note that $P$ is a projection operator defined on a Hilbert space whose operator norm is one. Now, we use the successive approximation method \cite[ Theorem A.6.4]{YukselBasarBook}. According to \eqref{eq:1.13}, we can write for  $k=1,2,...,N$
\begin{flalign*}
&M\gamma^{k*}_{N,(i)}(v^{k})+\epsilon\gamma^{k*}_{N,(i)}(v^{k})-\epsilon\gamma^{k*}_{N,(i)}(v^{k})+\bigg(\frac{Q}{N}-D\bigg)\nonumber\\&\ \times\bigg[\mathbb{E}(\underline{x}|v^{k})+\frac{1}{N}\sum_{p=1,p\not=k}^{N}\mathbb{E}\bigg(\gamma^{p*}_{N,(i)}(v^{p})|v^{k}\bigg)\bigg]=0.
\end{flalign*}
Thus, by dividing the expression over $\epsilon$ and rearranging it, we have
\begin{flalign*}
&\scalemath{0.95}{\gamma^{k*}_{N,(i)}(v^{k})=(1-\frac{\hat{R}_{ii}}{\epsilon})\gamma^{k*}_{N,(i)}(v^{k})-\frac{1}{\epsilon}\left(\frac{Q}{N}-D\right)}\\
&\scalemath{0.95}{\times \bigg[\mathbb{E}(x|v^{k})+\frac{1}{N}\sum_{p=1,p\not=k}^{N}\mathbb{E}\bigg(\gamma^{p*}_{N,(i)}(v^{p})\bigg|v^{k}\bigg)\bigg]},
\end{flalign*}
where the initial points of the iterations are zero functions. We can write 
$\underline{\gamma}_{N}^{*}(\underline{v})=P(I-\frac{1}{\epsilon}\underline{\hat{R}})\underline{\gamma}_{N}^{*}(\underline{v})-\frac{1}{\epsilon}Pr(\omega)$.
Similar to \cite[Theorem 2.6.5]{YukselBasarBook}, the above sequence converges to the unique fixed point if and only if the spectral radius satisfies the following constraint
$\rho\left(P(I-\frac{\underline{\hat{R}}}{\epsilon})\right)=\rho(I-\frac{\underline{\hat{R}}}{\epsilon}):=\lim\limits_{k\rightarrow \infty} \sup[||A||^{k}]^{\frac{1}{k}}<1$,
where $A:=\displaystyle I-\frac{\underline{\hat{R}}}{\epsilon}$, $||A||:=\sup\limits_{||x||<1}||Ax||$ and $\rho$ denotes spectral radius. The first equality is true since both $P$ and $A$ maps $\bf{\Gamma}_{N}$ into itself and $P$ has operator norm equal to one. The above constraint can be always satisfied by choosing $\epsilon=\frac{1}{2}(\lambda_{max}(\underline{\hat{R}})+\lambda_{min}(\underline{\hat{R}}))$. On the other hand, since $(\underline{x}, \underline{z}^{1},\dots,\underline{z}^{N})$ are jointly Gaussian, then $\gamma^{k*}_{N}(v^{k})=\pi^{k}_{N}v^{k}$ for $k=1,\dots,N$. Hence, $\gamma^{k*}_{N,(i)}(v^{k})=\pi^{k}_{N,(i)}v^{k}$, and by linearity of the conditional expectation, we have $\mathbb{E}(\underline{x}|v^{k})=S^{k}v^{k}$, and $\mathbb{E}(\gamma^{p*}_{N}(v^{p})|v^{k})=\pi^{p}_{N}{H^{p}S^{k}}v^{k}$.
Hence \eqref{eq:1.12} holds. Following from \cite{YukselBasarBook}, the stationary policy is globally optimal for ($\mathcal{P}_{N}$), and this completes the proof.
\end{proof}
\begin{theorem}
Consider $(\mathcal{P}_{\infty})$ with the expected cost \eqref{eq:1.9}. Under the following measurement scheme 
\begin{equation}
\underline{v}^{i}=H{\underline{x}}+\underline{z}^{i},\label{eq:19}
\end{equation}
where $\underline{z}^{i}$s are i.i.d. Gaussian random vectors, $\gamma^{i*}_{\infty}(v^{i})=\pi^{*}_{\infty}v^{i}$ is an optimal policy for ($\mathcal{P}_{\infty}$) and is the pointwise limit of $\gamma^{i*}_{N}(v^{i})=\pi^{*}_{N}v^{i}$, an optimal policy for ($\mathcal{P}_{N}$).
\end{theorem}
\begin{proof}
In the following, we invoke Proposition \ref{prop:4.7} and Theorem \ref{the:4.9.9} to prove the theorem. Under \eqref{eq:19}, the static team is symmetrically optimal and hence from \eqref{eq:1.12}, we have $\pi^{*}_{N}=L_{N}[S+{N}^{-1}({N-1})\pi_{N}^{*}HS]$, $\pi^{*}_{\infty}=R^{-1}D[S+\pi_{\infty}^{*}HS]$, where  $L_{N}:={(N^{2}R-2DN+Q)^{-1}}{(N^{2}D-NQ)}$, $S:={\Sigma_{00}(H)^{T}}({H\Sigma_{00}(H)^{T}+\Sigma_{kk}})^{-1}$. Since for every $N$, we have $J_{N}(\underline{\gamma}^{*}_{N})<\infty$, and since $R>0$, we have
$\sup\limits_{N\geq 1}\mathbb{E}(||\gamma^{*}_{N}(v^{1})||_{2}^{2})<\infty$,
which implies (A3). The proof is completed using the results of Proposition \ref{prop:4.7} and Theorem \ref{the:4.9.9}. One can also invoke Theorem \ref{lem:4.1} to justify the result.
\end{proof}
\subsubsection{Example 4, Asymmetric LQG team problems}\label{Ex:4.4}
Here, we consider simple variation of Example 3 considered above to illustrate Remark \ref{rem:asym}. Consider the observation scheme \eqref{eq:19}, and let the expected cost function be defined as
\begin{flalign*}
\scalemath{0.95}{J(\underline\gamma)}%&=\limsup\limits_{N\rightarrow \infty}J_{N}(\underline{\gamma})\\
&\scalemath{0.95}{=\limsup\limits_{N\rightarrow \infty} \frac{1}{N} \mathbb{E}^{\underline{\gamma}}\bigg[\sum_{i=1}^{N}(u^{i})^{T}Ru^{i}-2\sum_{i=1}^{N}(u^{i})^{T}D}\nonumber\\&\scalemath{0.95}{\times(\underline{x}+\frac{1}{N}\sum_{k=1}^{N}u^{k})+(\underline{x}+\frac{1}{N}\sum_{k=1}^{N}u^{k})^{T}Q(\underline{x}+\frac{1}{N}\sum_{k=1}^{N}u^{k})}\\&\scalemath{0.95}{+\frac{1}{N}\sum_{k=1}^{M}(u^{k})^{T}\alpha_{k}u^{k}\bigg]},
\end{flalign*}
where $M \in \mathbb{Z}_+$ is independent of $N$. Clearly, the $N$-DM team admits asymmetric optimal policies for $(\mathcal{P}_{N})$ with the expected cost $J_{N}$ for every $N$. However, one can observe that the last term above goes to zero as $N \to \infty$ under a sequence of optimal policies, and hence asymptotically the expected cost would essentially be \eqref{eq:1.9} and Theorem \ref{the:7} implies $\gamma^{*}_{\infty}$ is an optimal policy since $\mathbb{P}$-almost surely the sequence $Q_{N}$ converges weakly (the asymmetric term vanishes when $N \to \infty$). That is, the optimal policy designed for the symmetric problem is also a solution for the asymmetric problem since under this policy the additional term (which is a non-negative contribution) vanishes, certifying its optimality. 
\subsection{Example 5,  Multivariable classical linear quadratic Gaussian problems: average cost optimality through static reduction}\label{Ex:4}
Here, we revisit a well-known problem and a well-known solution, using the technique presented in this paper. Let
\begin{equation*}
X_{t+1}=AX_{t}+Bu^{t}+w^{t},
\end{equation*}
where $A \in \mathcal{M}_{n,n}(\mathbb{R})$, $B \in \mathcal{M}_{n,m}(\mathbb{R})$ and $w^{t}$s and $X_{0}$ are i.i.d. Gaussian random vectors with mean zero and positive variance taking values in $\mathbb{R}^{n}$.  Let $(A,B)$ be controllable and let
\begin{flalign*}
\scalemath{0.95}{J(\underline{\gamma})}&\scalemath{0.95}{=\limsup\limits_{T\rightarrow \infty}J_{T}(\underline{\gamma})}\\&\scalemath{0.95}{:=\limsup\limits_{T\rightarrow \infty}\frac{1}{T}\mathbb{E}^{\underline{\gamma}}\bigg[\sum_{t=0}^{T-1}X_{t}^{T}QX_{t}+(u^{t})^{T}Ru^{t}\bigg]},
\end{flalign*}
where $Q\geq 0$ and $R>0$ are appropriate dimensions real matrices. We can write,
\begin{flalign*}
\scalemath{0.95}{J(\underline{\gamma})=}&\scalemath{0.95}{\limsup\limits_{T\rightarrow \infty}\frac{1}{T}\mathbb{E}^{\underline{\gamma}}\bigg[\sum_{t=0}^{T-1}\bigg(\sum_{k=1}^{t}A^{t-k}Bu^{k-1}+\sum_{k=0}^{t}A^{t-k}\zeta^{k}\bigg)^{T}}\nonumber\\
&\scalemath{0.95}{\times Q\bigg(\sum_{k=1}^{t}A^{t-k}Bu^{k-1}+\sum_{k=0}^{t}A^{t-k}\zeta^{k}\bigg)+(u^{t})^{T}Ru^{t}\bigg]},
\end{flalign*}
where $\zeta=(X_{0}^{T},(w^{0})^{T},(w^{1})^{T},\dots)^{T}$. In the following, we consider fully observed classical IS, i.e., $Y^{t}=X_{t}$, and we can write $Y^{t}=H^{t}\zeta+\sum_{j=0}^{t-1}D_{tj}u^{j}$,
where $H^{t}$ and $D_{tj}$ are appropriate dimensional matrices. Using \cite[Theorem 1]{HoChu}, we can reduce IS to the static one as
%\begin{equation}
$V^{t}=\tilde{H}_{t}\zeta$.
%\end{equation}
%where 
%\begin{dmath*}
%K_{t,\infty}=
%\begin{bmatrix}
    %1       & 0 & 0 & 0 & 0 &\dots & 0 \\
    %a       & 1 & 0 & 0 &0 & \dots & 0 \\
      %a^{2}&  a^{1}      & 1 & 0& 0 & \dots & 0 \\
    %\hdotsfor{5} \\
    %a^{t}  &a^{t-1}&\dots &1&0  & \dots  & 0
%\end{bmatrix}
%\end{dmath*}
According to \cite[Section 3.5]{HernandezLermaMCP}, we have $u^{t*}_{T}=G^{t}_{T}X_{t}$ for $t=0,1,\dots, T-1$, where $k_{T}^{T}=0$, and
 \begin{equation}\label{eq:m.5.4.128}
 \scalemath{0.95}{G_{T}^{t}=-(R+B^{T}k_{T}^{t+1}B)^{-1}B^{T}k_{T}^{t+1}A},
  \end{equation}
   \begin{equation}\label{eq:m.5.129}
 \scalemath{0.95}{k_{T}^{t}=Q+A^{T}k_{T}^{t+1}A-A^{T}k_{T}^{t+1}B(R+B^{T}k_{T}^{t+1}B)^{-1}B^{T}k_{T}^{t+1}A}.
  \end{equation} %We have
 %\begin{equation}\label{eq:m.5.122}
 %\underline{u}_{T}^{*}=\underline{\pi}_{T}^{*}{v}^{T},
 %\end{equation}
% where
 %\begin{dmath*}
%\underline{\pi}_{T}^{*}=
%\begin{bmatrix}
    %G_{T}^{0}      &0 & 0 &\dots & 0&0 \\
    %G_{T}^{0}G_{T}^{1}       & G_{T}^{1} & 0&0  & \dots & 0 \\
     % G_{T}^{0}G_{T}^{2}(G_{T}^{1}+a)&  G_{T}^{2}G_{T}^{1}      & G_{T}^{2}& 0&0 & \dots  \\
         %G_{T}^{0}G_{T}^{3}(G_{T}^{2}+a)(G_{T}^{1}+a)  & G_{T}^{1}G_{T}^{3}(G_{T}^{2}+a)&  G_{T}^{3}G_{T}^{2}      & G_{T}^{3}& 0 & \dots  \\
    %\hdotsfor{5} \\
   %G_{T}^{T-1}G_{T}^{0}\prod_{p=1}^{T-2}(G_{T}^{p}+a)&\dots &G_{T}^{T-1}G_{T}^{T-3}(G_{T}^{T-2}+a)& G_{T}^{T-1}G_{T}^{T-2} & G_{T}^{T-1}
%\end{bmatrix}
%,
%\end{dmath*}
%where $G^{T-1}_{T}=0$. 
%\begin{remark}
%Since in classical information structure, information structure are nested we simplify $\underline{\pi}_{T}^{*}$ unless we require to show them as a vector with blocks of $\underline{\pi}_{T}^{t*}$ for $t=0,\dots,T-1$.
%\end{remark}
 \begin{theorem}\label{6.4}
For LQG teams with the classical information structure, ${u}^{t*}=\displaystyle\lim_{T \to \infty}\gamma^{t*}_{T}(v^{t})=\gamma^{t*}_{\infty}(v^{t})$\@  is the optimal policy for $J(\underline{\gamma})$\@, where $\{{\gamma}^{t*}_{T}\}_{T}$ is a sequence of optimal policies for $\{J_{T}(\underline{\gamma}_{T})\}_{T}$ with the pointwise limit ${\gamma}^{t*}_{\infty}$  as $T \to \infty$.
\end{theorem}
Although, this result is a classical one in the literature, here, we present a new approach using the static reduction.
\begin{proof}
Since,
% \begin{observation}
$ k^{t}_{T+1}=k^{t-1}_{T}$~for~$t=1,2,\dots,T$,
 %\end{observation}
one can write \eqref{eq:m.5.129} as
  \begin{flalign*}
\scalemath{0.91}{ k_{T}^{t}=Q+A^{T}k_{T-1}^{t}A-A^{T}k_{T-1}^{t}B(R+B^{T}k_{T-1}^{t}B)^{-1}B^{T}k_{T-1}^{t}A}.
  \end{flalign*}
 We use Theorem \ref{the:5} and Remark \ref{lem:1.1.1}, to show that $u^{t*}_{\infty}=G_{\infty} X_{t}$ is team optimal, where $G_{\infty}=-(R+B^{T}C^{*}B)^{-1}B^{T}C^{*}A$, and following from controllability of $(A,B)$, $C^{*}=\lim\limits_{\beta \rightarrow 1}C_{\beta}$, a fixed point of the following recursion exists, $C_{\beta}(n)=Q+A^{T}\beta C_{\beta}(n-1)A-A^{T}\beta C_{\beta}(n-1)B\left(R+B^{T}\beta C_{\beta}(n-1)B\right)^{-1}B^{T}\beta C_{\beta}(n-1)A$.
By comparing $C^{*}(n)=\lim\limits_{\beta \rightarrow 1}C^{*}_{\beta}(n)$ and \eqref{eq:m.5.129}, we have
$\lim\limits_{T \rightarrow \infty}k_{T}^{t}=K=C^{*}=\lim\limits_{n \rightarrow \infty}C^{*}(n)$.
Hence, for $t=0,1,\dots,T-1$,
$\lim\limits_{T \rightarrow \infty}G_{T}^{t}=-(R+B^{T}KB)^{-1}B^{T}KA=-(R+B^{T}C^{*}B)^{-1}B^{T}C^{*}A=G_{\infty}$.
Now, we use Remark \ref{lem:1.1.1} to show \eqref{eq:2.14} holds.
\begin{flalign}
&\scalemath{0.95}{\limsup\limits_{T\rightarrow \infty}\bigg|J_{T}(\underline{\gamma}_{T}^{*})-J_{T}(\underline{\gamma}_{\infty}^{*})\bigg|}\nonumber\\
%&\leq \limsup\limits_{T\rightarrow \infty}\sup_{0\leq t\leq T-1}|\mathbb{E}[\sum_{k=0}^{t}\zeta_{k}^{T}\left(L_{T}^{t,k}\right)^{T}(H_{T}^{t})\left(\sum_{k=0}^{t}L_{T}^{t,k}\zeta_{k}\right)]\nonumber\\
%&~~~~~-\mathbb{E}[\sum_{k=0}^{t}\zeta_{k}^{T}\left(L_{\infty}^{t,k}\right)^{T}(H_{\infty}^{t})\left(\sum_{k=0}^{t}L_{\infty}^{t,k}\zeta_{k}\right)]|\nonumber\\
&\scalemath{0.93}{\leq \limsup\limits_{T\rightarrow \infty}\sup_{0\leq t\leq T-1}\bigg|\mathbb{E}\bigg[\sum_{k=0}^{t}Tr\bigg(\zeta_{k}^{T}\bigg(L_{T}^{t,k}\bigg)^{T}(H_{T}^{t})L_{T}^{t,k}\zeta_{k}\bigg)\bigg]}\nonumber\\
&\scalemath{0.93}{-\mathbb{E}\bigg[\sum_{k=0}^{t}Tr\bigg(\zeta_{k}^{T}\bigg(L_{\infty}^{t,k}\bigg)^{T}(H_{\infty}^{t})L_{\infty}^{t,k}\zeta_{k}\bigg)\bigg]\bigg|}\label{eq:5.110}\\
&\scalemath{0.93}{\leq \limsup\limits_{T\rightarrow \infty}\sup_{0\leq t\leq T-1}\bigg|\mathbb{E}\bigg[\sum_{k=0}^{t}Tr\bigg(\zeta_{k}\zeta_{k}^{T}\bigg(\bigg(L_{T}^{t,k}\bigg)^{T}(H_{T}^{t})L_{T}^{t,k}}\nonumber\\&\scalemath{0.93}{-\bigg(L_{\infty}^{t,k}\bigg)^{T}(H_{\infty}^{t})L_{\infty}^{t,k}\bigg)\bigg)\bigg]\bigg|}\label{eq:5.112}\\
&\scalemath{0.93}{\leq \Sigma^{2}\limsup\limits_{T\rightarrow \infty}\sup_{0\leq t\leq T-1}\bigg|Tr\bigg((H_{T}^{t})C_{T}^{t}-(H_{\infty}^{t})C_{\infty}^{t}\bigg)\bigg|}\label{eq:5.5.52}\\
&\scalemath{0.93}{\leq \Sigma^{2}\limsup\limits_{T\rightarrow \infty}\bigg[\sup_{0\leq t\leq T-1}\bigg|Tr\bigg((G_{T}^{t})^{T}RG_{T}^{t}C_{T}^{t}}\nonumber\\&\scalemath{0.93}{-(G_{\infty})^{T}RG_{\infty}C_{\infty}^{t})\bigg|+\sup_{0\leq t\leq T-1}\bigg|Tr(Qe_{T}^{t})\bigg|\bigg]}\label{eq:5.5.53}\\
&\scalemath{0.93}{\leq \Sigma^{2}\limsup\limits_{T\rightarrow \infty}\bigg[\sup_{0\leq t\leq T-1}\bigg|Tr\bigg((G_{T}^{t})^{T}RG_{T}^{t}-(G_{\infty})^{T}RG_{\infty}\bigg)\bigg|}\nonumber\\&\scalemath{0.93}{\times\sup_{0\leq t\leq T-1}\bigg|Tr(C_{T}^{t})\bigg|}\nonumber\\
& \scalemath{0.93}{+\sup_{0\leq t\leq T-1}\bigg|Tr\bigg(G_{\infty}^{T}RG_{\infty}e_{T}^{t}\bigg)\bigg|+\sup_{0\leq t\leq T-1}\bigg|Tr(Qe_{T}^{t})\bigg|\bigg]}\label{eq:5.5.54}\\
&\scalemath{0.93}{\leq \Sigma^{2}\limsup\limits_{T\rightarrow \infty}\bigg[\sup_{0\leq t\leq T-1}\bigg|Tr\bigg((G_{T}^{t}(G_{T}^{t})^{T}-G_{\infty}G_{\infty}^{T})R\bigg)\bigg|}\nonumber\\&\scalemath{0.93}{\times\bigg(\sup_{0\leq t\leq T-1}\bigg|Tr(e_{T}^{t})\bigg|+\sup_{0\leq t\leq T-1}\bigg|Tr(C_{\infty}^{t})\bigg|\bigg)}\nonumber\\&\scalemath{0.93}{+\sup_{0\leq t\leq T-1}\bigg|Tr\bigg(G_{\infty}^{T}RG_{\infty}e_{T}^{t}\bigg)\bigg|+\sup_{0\leq t\leq T-1}\bigg|Tr(Qe_{T}^{t})\bigg|\bigg]=0}\nonumber,
\end{flalign}
where $L_{T}^{t,k}:=\prod_{p=k}^{t-1}(A+BG^{p}_{T})$, $L_{\infty}^{t,k}:=\prod_{p=k}^{t-1}(A+BG_{\infty})$, $H_{T}^{t}=(Q+(G^{t}_{T})^{T}RG^{t}_{T})$, $H_{\infty}^{t}=(Q+(G_{\infty})^{T}RG_{\infty})$, $e_{T}^{t}:=C_{T}^{t}-C_{\infty}^{t}$, and $C_{T}^{t}:=\left[\sum_{k=0}^{t}L_{T}^{t,k}\left(L_{T}^{t,k}\right)^{T}\right]$, $C_{\infty}^{t}:=\left[\sum_{k=0}^{t}L_{\infty}^{t,k}\left(L_{\infty}^{t,k}\right)^{T}\right]$ and $\Sigma^{2}:=\max(\sigma_{X_{0}}^{2}, \sigma_{w}^{2})$, where $\sigma_{X_{0}}^{2}$ and $\sigma_{w}^{2}$ are the variance of each component of $X_{0}$ and $w^{k}$, respectively. Equality \eqref{eq:5.110} follows from the fact that $\{w^{k}\}_{k}$ are i.i.d. and independent from $X_{0}$. Inequality \eqref{eq:5.112} follows from the trace property that $Tr(ABC)=Tr(CAB)$, and \eqref{eq:5.5.52} follows from the hypothesis that $\zeta_{k}$s are i.i.d. random vectors and $Tr(ABC)=Tr(BCA)$ and \eqref{eq:5.5.53} follows from linearity of the trace and $\sup f+g \leq \sup f+ \sup g$. Inequality \eqref{eq:5.5.54} follows from adding and subtracting $G_{\infty}^{T}RG_{\infty}C^{t}_{T}$ in the first term and using $Tr(AB)\leq Tr(A)Tr(B)$ for $A$ and $B$ positive semi-definite matrices since \eqref{eq:m.5.129} implies that for a fixed $T$, $\{k_{T}^{t}\}_{t=0}^{T-1}$ is a decreasing sequence, i.e., $K>k_{T}^{0}>k_{T}^{1}>\dots>k_{T}^{T-1}$, and hence $\{G_{T}^{t}(G_{T}^{t})^{T}\}_{t=0}^{T-1}$ is a decreasing sequence. Also, from \eqref{eq:m.5.4.128}, we have for a fixed $T$, $\{(A+BG_{T}^{t})(A+BG_{T}^{t})^{T}\}_{t=0}^{T-1}$ is an increasing sequence, hence, $(G_{T}^{t})^{T}RG_{T}^{t}-G_{\infty}^{T}RG_{\infty}$ is positive semi-definite. Finally, the last inequality follows from the definition of $e_{T}^{t}$ and the following calculation. First note that for a fixed $T$, $\{Tr(e_{T}^{t})\}_{t=0}^{T-1}$ is an increasing sequence.
%\begin{flalign*}
%&Tr(e_{T}^{t+1})-Tr(e_{T}^{t})\\
%&=Tr(\sum_{k=0}^{t+1}\left[L_{T}^{t+1,k}\left(L_{T}^{t+1,k}\right)^{T}-L_{\infty}^{t+1,k}\left(L_{\infty}^{t+1,k}\right)^{T}\right]\\&\ ~~~~~ -\sum_{k=0}^{t}\left[L_{T}^{t,k}\left(L_{T}^{t,k}\right)^{T}-L_{\infty}^{t,k}\left(L_{\infty}^{t,k}\right)^{T}\right])\\
%& =Tr(\sum_{k=0}^{t}[L_{T}^{t,k}\left((A+BG_{T}^{t})(A+BG_{T}^{t})^{T}-I\right)\left(L_{T}^{t,k}\right)^{T}]\\
%& ~~~~~-[L_{\infty}^{t,k}\left((A+BG_{\infty})(A+BG_{\infty})^{T}-I\right)\left(L_{\infty}^{t,k})^{T}]\right)\\
%& \geq Tr[\sum_{k=0}^{t}\left(L_{\infty}^{t,k}\right)^{T}L_{\infty}^{t,k}[(A+BG_{T}^{t})(A+BG_{T}^{t})^{T}\\&~~~~~-(A+BG_{\infty})\left(A+BG_{\infty}\right)^{T}]]\geq 0,
%\end{flalign*}
%where the inequalities follow from aforementioned observations and since $Tr(ABC)=Tr(CAB)$. 
Hence,
\begin{equation*}
\scalemath{0.95}{\lim\limits_{T\rightarrow \infty}\sup_{0\leq t\leq T-1}|Tr(e_{T}^{t})|=\lim\limits_{T\rightarrow \infty}|Tr(e_{T}^{T-1})|=0.}
\end{equation*}
Similarly, $\lim\limits_{T\rightarrow \infty}\sup_{0\leq t\leq T-1}|Tr(Qe_{T}^{t})|=0$. We have,%=\lim\limits_{T\rightarrow \infty}\sup_{0\leq t\leq T-1}|Tr\left((G_{\infty})^{T}RG_{\infty}e_{T}^{t}\right)|=0$.
\begin{flalign*}
&\scalemath{0.95}{\lim\limits_{T\rightarrow \infty}\sup_{0\leq t\leq T-1}|Tr(C_{\infty}^{t})|=\left|Tr\left[(I-(A+BG_{\infty}))^{-1}\right]\right|=0},
%&\ =\lim\limits_{T\rightarrow \infty}\sup_{0\leq t\leq T-1}\left|Tr[(I-(A+BG_{\infty})\right)^{-1}\\&~~~~~\times\left(I-(A+BG_{\infty})(A+BG_{\infty})^{T}\right)^{(t+1)}]|\\
%&\ =\lim\limits_{T\rightarrow \infty}|Tr[\left(I-(A+BG_{\infty})\right)^{-1}\\&~~~~~\times\left(I-(A+BG_{\infty})(A+BG_{\infty})^{T}\right)^{(T)}]|\\,
\end{flalign*}
where $Y^{(T)}$ denotes the $T$ power of the matrix $Y$ and the result follows from $||A+BG_{\infty}||<1$ (following from the controllability assumption). Finally, we have
\begin{flalign*}
&\scalemath{0.95}{\lim\limits_{T\rightarrow \infty}\sup_{0\leq t\leq T-1}\left|Tr\left[\left(G_{T}^{t}(G_{T}^{t})^{T}-G_{\infty}G_{\infty}^{T}\right)R\right]\right|=0},
\end{flalign*}
 where the second equality follows from the aforementioned observations and since $R$ is positive definite. Therefore, $\limsup\limits_{T\rightarrow \infty}|J_{T}(\underline{\gamma}_{T}^{*})-J(\underline{\gamma}^{*}_{\infty})|=0$, and the proof is completed.
\end{proof}
\begin{remark}
Similarly, one can show the result for (i) $Y^{t}=CX^{t}$, $(A,C)$ is observable and $Q=C^{T}C$, (ii) the discounted LQG team problems with the classical information structure.
\end{remark}
\section{Conclusion}\label{sec:6}
%In this paper, sufficient conditions for global optimality concerning with static teams with countably infinite decision makers have been presented. For a general setup of static team problems, a conclusive result has been established leads to sufficient conditions under which optimal ergodic costs of static teams with $N$ decision makers converges to the corresponding optimal cost of static teams with countably infinite decision makers. Under uniform integrability and convergence, it has been established that team optimal policies for static team with countably infinite decision makers is a limit-point of sequence of optimal policies for static teams with $N$ decision makers. Relaxed sufficient conditions has been reported under which optimal policies of static symmetric teams with countably infinite decision makers identify as a limiting points of sequences of symmetric optimaly policies of static teams with $N$-decision makers. Linear quadratic and  and linear quadratic Gaussian static teams with countably many decision makers have been studied and sufficient conditions has been introduced under the optimally symmetric assumption.
In this paper, we studied static teams with countably infinite number of DMs. We presented sufficient conditions for team optimality concerning average cost problems. Also, constructive results have been established to obtain the team optimal solution for static teams with countably infinite number of DMs as limits of the optimal solutions for static teams with finite number of DMs as the number of DMs goes to infinity. We also studied sufficient conditions for team optimality of symmetric static teams and mean-field teams under relaxed conditions. We recently studied convex dynamic teams with countably infinite DMs and mean-field teams \cite{sanjari2019optimal}.

\section*{Acknowledgments}
The authors are grateful to Prof. Naci Saldi for the detailed suggestions he has provided, and to Dr. Ather Gattami, Prof. Nuno Martins and Prof. Aditya Mahajan for related discussions on this topic.
%
% \Cref{ex:secref} shows how to reference sections.

% \begin{example}[label={ex:secref},lefthand ratio=0.4,bicolor,sidebyside,%
% listing options={style=siamlatex,basicstyle=\ttfamily\scriptsize,%
% deletetexcs={cref,Cref},{moretexcs=[2]{cref,Cref}}}]
% {Right and wrong ways to reference a section}
% Inside a sentence\dots\\
% Single: \cref{sec:intro}\\
% Range: \cref{sec:intro,sec:front,%
% sec:sec}\\
% Multiple: \cref{sec:intro,sec:sec,%
% sec:tab,sec:math,sec:thm}\\
% Appendix: \cref{sec:changes}\\

% Beginning of a sentence\dots\\
% Single: \Cref{sec:intro}\\
% Range: \Cref{sec:intro,sec:front,%
% sec:sec}\\
% Multiple: \Cref{sec:intro,sec:sec,%
% sec:tab,sec:math,sec:thm}\\
% Appendix: \Cref{sec:changes}\\

% Just don't do it this way\dots\\
% Section~\ref{sec:intro}  
% \end{example}

\bibliographystyle{plain}

\begin{IEEEbiography}
{\bf Sina Sanjari} received his M.Sc. degree in Electrical and Computer Engineering from Tarbiat Modares University and M.Sc. degree in Mathematics and Statistics from Queen's Uuniversity in 2015 and 2017, respectively. He is now Ph.D. student in Department of Mathematics and Statistics, Queen's university.  His research interests are on stochastic control, decentralized control, mean-field games, and applications.
\end{IEEEbiography}

\begin{IEEEbiography}
{\bf Serdar Y\"uksel} (S'02, M'11) received his B.Sc. degree in Electrical and Electronics
Engineering from Bilkent University in 2001; M.S. and Ph.D. degrees in
Electrical and Computer Engineering from the University of Illinois at Urbana-
Champaign in 2003 and 2006, respectively. He was a post-doctoral researcher
at Yale University before joining Queen's University as an Assistant
Professor in the Department of Mathematics and Statistics, where he is now an Associate Professor. He has been awarded
the 2013 CAIMS/PIMS Early Career Award in Applied Mathematics. His
research interests are on stochastic control, decentralized control, information
theory and probability. He is an Associate Editor for the IEEE TRANSACTIONS ON AUTOMATIC CONTROL and Automatica.
\end{IEEEbiography}

\end{document}